\documentclass[12pt]{amsart}

\usepackage[utf8]{inputenc}

\usepackage{fancyhdr}

\usepackage{amsmath, amssymb,amsthm, amsfonts}

\usepackage[normalem]{ulem}

\def\n{\nabla}

\def\bee{\begin{equation*}}
\def\eee{\end{equation*}}

\def\e{\epsilon}

\def\p{\partial}
\newcommand\R{{\mathbb R}}

\newcommand\C{{\mathbb C}}

\def\K{K\"ahler }

\def\KRF{K\"ahler-Ricci flow }

\def\A{Amp\`{e}re }

\def\ddbar{\partial\bar\partial}
\def\be{\begin{equation}}
\def\ee{\end{equation}}

\def\re{\text{Re}}

\def\e{\epsilon}

\def\Ric{\text{\rm Ric}}

\def\p{\partial}
\def\re{\text{\rm Re}}

\def\C{\Bbb C}

\def\p{\partial}
\def\re{\text{\rm Re}}

\def\p{\partial}

\def\C{\Bbb C}

\def\ddbar{\partial\bar\partial}
\def\ddb{\partial\bar\partial}
\def\KRF{K\"ahler-Ricci flow }

\def\n{\nabla}

\makeatletter
\def\l@subsection{\@tocline{2}{0pt}{2.5pc}{5pc}{}}
\def\l@subsubsection{\@tocline{3}{0pt}{3pc}{5pc}{}}
\makeatother

\newtheorem{thm}{Theorem}[section]
\newtheorem{ass}{Assumption}
\newtheorem{lem}{Lemma}[section]
\newtheorem{prop}{Proposition}[section]

\theoremstyle{definition}
\newtheorem{defn}{Definition}[section]
\theoremstyle{remark}
\newtheorem{rem}{Remark}
\newtheorem{ack}{Acknowledgements}
\numberwithin{equation}{section}
 
\begin{document}
\title{K\"ahler-Ricci flow of cusp singularities on quasi projective varieties}
\author{Albert Chau$^1$, Ka-Fai Li, Liangming Shen}
\address{Department of Mathematics,
The University of British Columbia, Room 121, 1984 Mathematics
Road, Vancouver, B.C., Canada V6T 1Z2} 
\email{chau@math.ubc.ca}  \email{kfli@math.ubc.ca}   \email{lmshen@math.ubc.ca}

\thanks{$^1$Research
partially supported by NSERC grant no. \#327637-06}

\thanks{\begin{it}2000 Mathematics Subject Classification\end{it}.  Primary 53C55, 58J35.}

\maketitle\markboth{Albert Chau, Ka-Fai Li, Liangming Shen} {Cusp flow}

\tableofcontents

\addcontentsline{toc}{section}{Abstract}

\begin{abstract}
Let $\overline{M}$ be a compact complex manifold with smooth \K metric $\eta$,  and let $D$ be a smooth divisor on $\overline{M}$.   Let $M=\overline{M}\setminus D$ and let $\widehat{\omega}$ be a Carlson-Griffiths type metric on $M$.   We study complete solutions to \K Ricci flow \eqref{ckrf} on $M$ which are comparable to $\widehat{\omega}$, starting from a smooth initial metric $\omega_0=\eta +i\ddbar \phi_0$ where $\phi_0\in C^{\infty}(M)$.  When $\omega_0\geq  c \widehat{\omega}$ on $M$ for some $c>0$ and $\phi_0$ has zero Lelong number, we construct a smooth solution $\omega(t)$ to \eqref{ckrf} on $M\times [0, T_{[\omega_0 ]})$ where $T_{[\omega_0 ]}:= \sup \{ T: [\eta] +T (c_1(K_{\overline{M}}) + c_1(\mathcal{O}_D))\in  \mathcal{K}_M \}$ so that $\omega(t)\geq (\frac{1}{n} -  \frac{4\hat{K}t}{c} )\widehat{\omega}$
for all $t\leq \frac{c}{4n\hat{K}}$  where $\hat{K}$ is a non-negative upper bound on the bisectional curvatures of $\widehat{\omega}$ (see Theorem \ref{T1}).  In particular, we do not assume $\omega_0$ has bounded curvature.  If $\omega_0$ has bounded curvature and is asymptotic to $\widehat{\omega}$ in an appropriate sense, we construct a complete bounded curvature solution on $M\times [0, T_{[\omega_0 ]})$ (see Theorem \ref{T3}).   These generalize some of the results of Lott-Zhang in \cite{LZ}.  On the other hand if we only assume $\omega_0\geq  c \eta$ on $M$ for some $c>0$ and $\phi_0$ is bounded on $M$, we construct a smooth solution to \eqref{ckrf} on $M\times [0, T_{[\omega_0 ]})$ which is equivalent to $\widehat{\omega}$ for all positive times.  This includes as a special case when $\omega_0$ is smooth on $\overline{M}$ in which case the solution becomes instantaneously complete on $M$ under \eqref{ckrf} (see Theorem \ref{T-bdd data}).

\end{abstract}

\section{Introduction}

 Let $\overline{M}$ be a compact complex manifold with smooth \K metric $\eta$.  Let $D$ be a smooth divisor on $\overline{M}$  and let $M=\overline{M}\setminus D$.  In this paper we consider the \K Ricci flow

\be\label{ckrf}
 \left\{
   \begin{array}{ll}
     \displaystyle\frac{\partial \omega(t)}{\partial t}=-Ric(\omega(t)) \\
     \omega(0)  = \omega_0.
   \end{array}
 \right.
 \ee
on $M$ for initial data having the form $\omega_0=\eta+i\partial\bar{\partial} \varphi_0$ where $\varphi_0 \in Psh(\overline{M}, \eta)$.  We refer to Definition \ref{1stdefn} for the notations used here and throughout the rest of the paper.   In particular, we also fix some holomorphic section $S$ of $\mathcal{O}_D$ vanishing precisely along $D$.

 We will be interested in constructing solutions to \eqref{ckrf} which are complete on $M$.  A canonical class to work with are the cusplike merics on $M$, which are metrics equivalent to the standard complete local model $$\displaystyle \frac{i dz^1 \wedge dz^{\bar{1}}}{|z^1|^2 \log^2 |z^1|^2} + i \sum_{j=2}^{n} dz^j \wedge dz^{\bar{j}}.$$
 in any holomorphic coordinate neighborhood $N$ of $\overline{M}$ where $D\bigcap N=\{z_1=0\}$.   In \cite{LZ} the authors showed that if $\omega_0$ is cusplike with so called {\it super standard spatial asymptotics} at $D$, then a bounded curvature cusplike solution $\omega(t)$ to \eqref{ckrf} exists on $M\times[0, T_{[\omega_0 ]})$, having the same asymptotics for all $t$, where

\begin{equation}\label{existenceestimate1}T_{[\omega_0 ]}:= sup \{ T: [\eta] +T (c_1(K_{\overline{M}}) + c_1(\mathcal{O}_D))\in  \mathcal{K}_{\overline{M}} \}\end{equation}
They also showed under a weaker condition called {\it standard spatial asymptotics}, a similar bounded curvature solution exists on $M\times[0, T)$, for some maximal $T$ where $T\leq T_{[\omega_0 ]}$.  A main point here is that the maximal existence time $T_{[\omega_0 ]}$ depends only on the cohomology class of the initial form $\omega_0$.  When $\omega_0$ has bounded curvature and is cusplike and satisfies a condition weaker than {\it super standard spatial asymptotics} at $D$, then we show there still exists a solution to \eqref{ckrf} on $M\times [0, T_{[\omega_0 ]})$ which is cusplike for all positive times (see Theorem \ref{T3} and below for details).  More generally, assuming $\omega_0$ is only bounded below on $M$ by a cusplike metric, and has zero Lelong number, we will construct a solution to \eqref{ckrf} on $M\times [0, T_{[\omega_0 ]})$ which is likewise bounded below by a cusplike metric on some definite positive time subinterval of $[0, T_{[\omega_0 ]})$ (see Theorem \ref{T1} and below for details).  In particular, $\omega_0$ may have unbounded curvature on $M$ here.  On the other hand, in cases when $\omega_0$ may be incomplete on $M$, including when $\omega_0$ is smooth on $\overline{M}$, we can still construct solutions on $M\times [0, T_{[\omega_0 ]})$ which are cusplike on $M$ for all positive times (see Theorem \ref{T-bdd data}).  In these cases $\omega_0$ becomes instantaneously complete on $M$ under \eqref{ckrf}.  We now describe our main results in more details below.

 We will study \eqref{ckrf} through an associated parabolic Monge \A equation set up as follows.  For any Hermitian metric $h$ on $\mathcal{O}_D$, and volume form $\Omega$ on $\overline{M}$, consider a solution $\varphi(t)$ to the parabolic Monge \A equation

\begin{equation}\label{pma0}
\begin{split}
\left\{
   \begin{array}{ll}
\partial_{t}\varphi(t)  = \displaystyle \log\frac{\|S\|_{h}^{2}\log^2\|S\|_{h}^{2}(\theta_{t}+i\partial\bar{\partial}\varphi(t))^{n}}{\Omega};\\
\varphi(0) =  \varphi_0.
  \end{array}
 \right.\\
\end{split}
\end{equation}
$$\,\,\,\,\,\,\,\,\, \theta_{t}:=\eta+t\chi; \hspace{12pt} \chi:=-\Ric(\Omega)+\Theta_{h}-i\partial\bar{\partial}\log\log^{2}\|S\|_{h}^{2}$$
and the associated family of \K metrics
\begin{equation}\label{krfansatz}\omega(t):= \theta_{t} +i\ddbar \varphi(t)\end{equation}
on $M\times[0, T)$ for some $T$.  It follows $\omega(t)$ solves \eqref{ckrf} on  $M\times[0, T)$, and conversely, if $\omega(t)$ solves \eqref{ckrf} on $M\times[0, T)$ then \eqref{krfansatz} holds for some solution $\varphi(t)$ to  \eqref{pma0} (see the derivation of \eqref{pma-smooth-appr}).  The equation \eqref{pma0} is different from the parabolic Monge \A equations considered in the earlier works mentioned above in the appearance of the $\|S\|^2\log \|S\|_h^2$ term in the numerator of the right hand side.  This term will be useful in establishing the cusplike-ness of our solutions for positive times.

 We first consider the case $\omega_0\geq c\eta$ on $M$ for some $c>0$ where $\varphi_0$ is bounded and smooth on $M$.  In particular, $\omega_0$ is typically incomplete on $M$ here.  Our main result here is

\begin{thm}\label{T-bdd data}
 Let  $\varphi_{0}\in L^{\infty}(\overline{M})\bigcap C^{\infty}(M) \bigcap Psh(\overline{M}, \eta)$ such that
\begin{equation}\label{lower}
\omega_0=\eta+ i \partial \bar{\partial} \varphi _0\geq c\eta\end{equation} for some constant $c>0$. Let  $T_{[\omega_0 ]}$ be as in \eqref{existenceestimate1}.    Then \eqref{ckrf} has a unique smooth solution $\omega(t)$ on $M\times[0, T_{[\omega_0 ]})$
 where \begin{equation}\label{completeness}c_1(t)\widehat{\omega}\leq  \omega_{t}\leq c_2(t)\widehat{\omega}\end{equation} for all $t\in (0, T_{[\omega_0 ]})$ ,and some positive functions $c_i(t)$ and Carlson-Griffiths form (see \S 2) on  $\widehat{\omega}$ on $M$.

 Also, for any hermitian metric $h$ on $\mathcal{O}_D$ and volume form $\Omega$ on $\overline{M}$, \eqref{pma0} and \eqref{krfansatz} hold on $M\times[0, T_{[\omega_0 ]})$ for some $\varphi(t)$ which is bounded on $M$ for each $t\in [0, T_{[\omega_0 ]})$.
\end{thm}

\begin{rem} Theorem \ref{T-bdd data}  includes as a special case, when $\omega_0$ has conical singularities at $D$ or is in fact smooth on $\overline{M}$.
\end{rem}
 
Next we consider when $\omega_0$ is a complete metric on $M$ in which case $\varphi_0$ may be unbounded on $M$.  Our first result here is

\begin{thm}\label{T1}
 Let $\varphi_0\in C^{\infty}(M)\bigcap Psh(\overline{M}, \eta)$ have zero Lelong number such that
$$\omega_0=\eta+ i \partial \bar{\partial} \varphi _0\geq c\widehat{\omega}$$ for some $c>0$ and Carlson-Griffiths form $\widehat{\omega}$ on $M$.  Let $T_{[\omega_0 ]}$ be as in \eqref{existenceestimate1}.  Then the \K-Ricci flow \eqref{ckrf} has a smooth solution $\omega(t)$ on $M\times[0, T_{[\omega_0 ]})$  and
\begin{equation}\label{completeness0}
\omega(t)\geq (\frac{1}{n} -  \frac{4\hat{K}t}{c} )\widehat{\omega}\end{equation}
 for all $t\leq \frac{c}{4n\hat{K}}$  where $\hat{K}$ is a non-negative upper bound on the bisectional curvatures of $\widehat{\omega}$.  Moreover,
\begin{enumerate}
\item For any hermitian metric $h$  on $\mathcal{O}_D$  and  volume form $\Omega$ on $\overline{M}$, \eqref{pma0} and \eqref{krfansatz} hold  on $M\times[0, T_{[\omega_0 ]})$ where $\varphi(t)\leq c(t)$ on $M\times  [0, T_{[\omega_0 ]})$ for some continuous function $c(t)$.
\item Suppose further that $\omega_0$ is cusplike and $-C\log\log^2 \|S\|^{2}_{h}\leq \varphi_0$ on $M$ for some constant $C>0$, Hermitian metric $h$.  Then for any $0<T<T_{[\omega_0 ]}$, \eqref{pma0} and \eqref{krfansatz} hold on $M\times[0, T]$ where $-c\log\log^2 \|S\|^{2}_{h'}\leq \varphi(t)\leq c$ on $M\times[0, T]$ for some constant $c>0$ and Hermitian metric $h'$.
\end{enumerate}
\end{thm}

\begin{rem}
It can be proved that there is a unique solution $g(t)$ satisfying (\ref{completeness0}) and (1) on the time interval $[0,\frac{c}{4n\hat{K}})$, though it is not known if our solution is unique such on the whole time interval $[0,T_{[\omega_0]})$.  Uniqueness of complete bounded curvature solutions to the real Ricci flow (in particular \eqref{ckrf}) was proved in \cite{Chen-Zhu} and for a more general class of complete solutions to \eqref{ckrf} in \cite{Chau-Li-Tam3}.
\end{rem}

Note also that Theorem \ref{T1}  leaves open the possibility that the solution may exist beyond $t=T_{[\omega_0 ]}$.  Note also, in Theorem \ref{T1} $\omega_0$ is complete while the solution may not be complete for all positive times.  Meanwhile in Theorem \ref{T-bdd data}, $\omega_0$ may be incomplete while the solution is complete for all positive times.  This seems counterintuitive, and is a result of the stronger a priori estimates in the case $\varphi_0$ is bounded.  On the other hand, (2) says cusplikeness is preserved at the potential level  for all times in some sense (see \S 2.1). If we assume $\omega_0$ above in fact has bounded curvature and is sufficiently asymptotic to the standard model at $D$ in a sense, the following theorem says the solution is indeed cusplike for all times, and $[0, T_{\omega})$ is indeed a maximal time interval.

\begin{thm} \label{T3}
 Let $\eta$ be a smooth \K form on $\overline{M}$ and $\widehat{\omega}=\eta-i\ddbar \log\log^2\|S\|^2$ be a Carlson-Griffiths form on $M$.  Let $\omega_0=\widehat{\omega}+i\ddbar \varphi$ be a smooth complete bounded curvature \K metric on $M$ such that $\frac{\varphi}{\log\log^2\|S\|^2}\to 0$, $\frac{|d\varphi|_{\widehat{\omega}}}{\log\log^2\|S\|^2}\to 0$ and $|\omega_0-\widehat{\omega}|_{\widehat{\omega}}\to 0$ as $\|S\|\to 0$.   Let $T_{[\omega_0 ]}$ be as in \eqref{existenceestimate1}. Then \eqref{ckrf} has a unique smooth maximal bounded curvature solution $\omega(t)$ on $M\times[0, T_{[\omega_0 ]})$.
\end{thm}

\begin{rem}  In \cite{LiuZhang} the authors studied \eqref{ckrf} in the case when $T_{[\omega_0 ]}=\infty$ in \eqref{existenceestimate1} and the initial data $\omega_0$ is a Carlson-Griffiths metric (see \S 2.1) which is a certain limit of metrics with conical singularities at $D$.   They constructed a solution $\omega(t)$ which in fact satisfies

\be\label{currentkrf}
 \left\{
   \begin{array}{ll}
     \displaystyle\frac{\partial \omega(t)}{\partial t}=-Ric(\omega(t))-\omega(t) +[D] \\
     \omega(0)  = \omega_0.
   \end{array}
 \right.
 \ee
in the sense of currents on $\overline{M}\times[0, \infty)$ where $[D]$ is the current of integration along $D$.  By the use of \eqref{pma0}, \eqref{krfansatz}, we can likewise show that if $\omega(t)$ is from Theorem \ref{T3} then $e^{-t} \omega(e^t-1)$ will solve \eqref{currentkrf} in the sense of currents on $\overline{M}\times[0, \log (T_{[\omega_0 ]} +1)$.  Here we have used the fact that $\omega(t)$ is equivalent to  Carlson-Grifiths metric and has uniformly bounded curvature on compact time intervals of $M\times[0, T_{[\omega_0 ]})$.
\end{rem}

\begin{rem}
In our results above we only considered the case of a single smooth divisor $D \subset \overline{M}$.  On the other hand, straight forward extensions of our definitions and techniques allow us also to consider the case of some collection of simple normal crossing divisors $D_1,..,D_k$ in which case we can have similar statements as in our main theorems for \K metrics on $M$ which are cusplike at each $D_i$.
\end{rem}

 The \K Ricci flow \eqref{ckrf} has been studied in several earlier works for singular initial metrics on compact \K manifolds.  We describe some of those results which in some apsects can be compared to our reuslt, and we refer therein for further references. The flow was studied for general singular initial metrics on compact \K manifolds in \cite{ST, GZ, DiLu}.  In the results there a solution to \eqref{ckrf} was constructed under the hypothesis in Theorem \ref{T1} where the solution is smooth on $\overline{M}$ for all positive times and converges to the initial data at time $0$ in some weak sense.  In the case, the solution we construct is different since as it is complete on $M$ for some positive times.   In \cite{CTZ}, solutions having conical singularities at $D$ were introduced and studied, and in \cite{Shen1, Shen2} it was shown that these solutions exist up to a maximal time given by \eqref{existenceestimate1} but with a factor  $(1-\beta)$ added in front of the term $\mathcal{O}_D$ where $\beta$ is the cone angle along $D$.  In particular, this factor is $1$ when letting $\beta =0$ in which case we recover the formula in \eqref{existenceestimate1} and this is consistent with our results if we view cusps as cones with angle $\beta=0$.  

\begin{ack} The first author would like to thank John Lott, Slawomir Kołodziej and Luen-Fai Tam for useful conversations.  The third author would also like to thank Gang Tian and Zhou Zhang for useful conversations.
\end{ack}

\section{Preliminaries}

Let $\overline{M}$ be a compact complex manifold with smooth \K metric $\eta$ and smooth divisor $D$, and let $M=\overline{M}\setminus D$ as in the introduction.  Throughout the paper we will use the following notation and definitions.

\begin{defn}\label{1stdefn} 
We will adopt the following notations and definitions

\vspace{5pt}

$\mathcal{K}_M$: the space of \K classes on $\overline{M}$

\vspace{5pt}

  $K_M$: the canonical bundle on $\overline{M}$
\vspace{5pt}

 $c_1(L)$: the first Chern class of a holomorphic line bundle  $L$ on $\overline{M}$

\vspace{5pt}

 $[\sigma]$: the $H^{1,1}(\overline{M},\mathbb{R})$ cohomology class of a 1-1 current $\sigma$ on $\overline{M}$

\vspace{5pt}

$\mathcal{O}_D$: the holomorphic line bundle associated to $D$

\vspace{5pt}

$S$: a holomorphic section of $\mathcal{O}_D$ vanishing exactly at $D$.

\vspace{10pt}

 $\Theta_h$: the curvature form of a hermitian metric $h$ on $\mathcal{O}_D$

($\Theta_h=-i\ddbar \log h$ locally on $\overline{M}$, and $\Theta_h=-i\ddbar \log \|S\|^2_h$ globally on $M$)

\vspace{10pt}
$Psh(\overline{M}, \eta)$: the set of plurisubharmonic functions on $\overline{M}$ relative
 to  

 $\eta$ (see definition \ref{psh} for details)

\end{defn}

\subsection{Carlson-Griffiths forms} Let $h$ be a Hermitian metric on $\mathcal{O}_D$.  Consider the Carlson-Griffiths type form on $M=\overline{M}\setminus D$  associated to $\eta, h$: \begin{equation}\label{CGmetric}\begin{split}\widehat{\omega}_{\eta, h}:=&\overline{\eta} - i\ddbar (\log\log^2 \|S\|^{2}_{h})\\
&=\overline{\eta} -2i\frac{\ddbar \log \|S\|^{2}_{h}}{\log \|S\|^{2}_{h}}+2i \frac{\partial \log \|S\|^{2}_{h}}{\log \|S\|^{2}_{h}}\wedge \frac{\bar{\partial} \log \|S\|^{2}_{h}}{\log \|S\|^{2}_{h}}\\
\end{split}\end{equation} as introduced  in \cite{CG72}.  We now observe some facts about Carlson-Griffiths forms and we refer to \cite{CG72}  (see also \cite{Guenancia}) for more details and explanations.

First, from the last line in \eqref{CGmetric} we may scale $h$ so that $\widehat{\omega}_{\eta, h}$ is positive on $M$ in which case we will refer it as a Carlson-Griffiths type metric on $M$.  Furthermore, around each point $p\in D$ there are holomorphic coordinates $z_1,..,z_n$ were $D=\{z_1=0\}$ in which case $\widehat{\omega}_{S, h}$ is equivalent to the local  model $$\displaystyle \frac{i dz^1 \wedge dz^{\bar{1}}}{|z^1|^2 \log^2 |z^1|^2} + i \sum_{j=2}^{n} dz^j \wedge dz^{\bar{j}}.$$
In particular Carlson-Griffiths type metrics are complete on $M$, any two are equivalent on $M$, and they also satisfy

\begin{enumerate}
\item $\hat{ \omega}_{\eta, S, h}$ has bounded geometry of infinite order.
\item $-\log\log^2 \|S\|^{2}_{h}$ is bounded above and in $L^1(\overline{M})$
\item $\log\frac{\widehat{\omega}_{S, h}^{n}\|S\|_{h}^{2}\log^{2}\|S\|_{h}^{2}}{\Omega}$ is bounded on $M$ where $\Omega$ is any smooth volume form on $\overline{M}$
\end{enumerate}
 In particular (2)  implies that $\widehat{\omega}_{\eta, h}$ is a well defined current on $\overline{M}$. (see for example \cite{LZ} (\S8, example 8.15)).   

\subsection{Approximation of functions in $Psh(\overline{M}, \eta)$}
Recall the following
\begin{defn}\label{psh}
 $\varphi$ is called $plurisubharmonic$ on the compact \K manifold $(\overline{M}, \eta)$, written $\varphi\in Psh(\overline{M}, \eta)$, when $\varphi:\overline{M}\to \R$ is upper semi-continuous and bounded above, and for any local holomorphic coordinate domain $U_{\alpha}$, $\eta_{\alpha}+\varphi$ is a classical plurisubharmonic function in $U_{\alpha}$ where $\eta_{\alpha}$ is a local \K potential.  In this context $\varphi\in C^{\infty}(M)\bigcap Psh(\overline{M}, \eta)$ is said to have $zero$ $Lelong$ $number$ if for any $c>0$ we have $$\lim_{d(x, D) \to 0} \frac{\varphi(x)}{ c\log \|S\|_h } \to 0$$ where $d(\cdot, D)$ is the distance to $D\subset M$ relative to $\eta$.
\end{defn}

Let $\varphi\in Psh(\overline{M}, \eta))$ be given.  By \cite{Demailly}, or \cite{Blocki-Koloziej} for a simpler proof in our setting, there exists a decreasing sequence $\varphi_j \in  C^{\infty}(\overline{M})\bigcap Psh(\overline{M}, \eta)$ converging pointwise to $\varphi$.  By a slight modification of the proof in \cite{Blocki-Koloziej}, we may assume this convergence actually holds in $C_{loc}^{\infty}(M)$ when $\varphi\in C^{\infty}(M)$.  We include the statement and proof of this below for completeness.

\begin{thm}\label{BK}
 Suppose that $\varphi \in Psh(\overline{M},\eta)\cap C^{\infty}(M)$. 
Then there exists a sequence $\varphi_j \in C^{\infty}(\overline{M})$ with $\varphi_j \downarrow \varphi$ pointwise on $\overline{M}$ and smoothly uniformly on compact subsets of $M$.
\end{thm}

 \begin{proof}
Let $S_j$ be an increasing sequence of open sets exhausting $M$ where each $\overline{S_j}$ is compact.  Let $m_j=\min \{\varphi(x): x\in \overline{S_j}\}$ and define $\varphi_j= \max\{\varphi,m_j\}+\frac{1}{j}$.  Then $\varphi_j\in C^0(\overline{M})\bigcap C^{\infty}(S_j)\bigcap Psh(\overline{M},\eta)$ with $\varphi_j \downarrow \varphi$ pointwise on $\overline{M}$ and smoothly uniformly on compact subsets. Now for each $j$, suppose there exists a sequence $\varphi_{j,k}\in Psh(\overline{M},\eta)\cap C^{\infty}(\overline{M})$ with $\varphi_{j,k}\downarrow \varphi_j$ pointwise uniformly on $\overline{M}$ and smoothly uniformly on $S_{j-1}$.  Then for any diverging sequence $a_j$, we have $\varphi_{j, a_j}\downarrow \varphi$ pointwise on $\overline{M}$ and smoothly uniformly on compact subsets.  Moreover, by the fact  $\varphi_j-\varphi_{j+1}\geq \frac{1}{j}-\frac{1}{j+1}$, it is clear that we may choose some sequence $a_j$ with $\varphi_{j, a_j}\downarrow \varphi$.

 From the above, it suffices now to prove the following
 \vspace{10pt}

CLAIM: if $\varphi\in C^0(\overline{M})\bigcap C^{\infty}(U)\bigcap Psh(\overline{M},\eta)$ for some open set $U$ and $V$ is a precompact open subset of $U$, there exists a sequence $\varphi_j \in C^{\infty}(\overline{M})\bigcap  Psh(\overline{M},\eta)$ with $\varphi_j \downarrow \varphi$ pointwise on $\overline{M}$ and smoothly uniformly on $V$.
 \vspace{10pt}

 The claim follows from a slight modification of the proof of the main Theorem in \cite{Blocki-Koloziej}.  In \cite{Blocki-Koloziej}, an arbitrary open cover $U_{\alpha}$ of $\overline{M}$ is first chosen, then in each $U_{\alpha}$ a smooth local approximation $\varphi_{\alpha,\delta}$ of $\varphi$ is constructed through the use of local \K potentials  and mollification.  Then for fixed $\delta$, a global smooth approximation of $\varphi$ on $\overline{M}$ is defined as the pointwise regularized maximum (from \cite{Demailly}) of the $\varphi_{\alpha,\delta}$'s where the maximum is taken over all $\alpha$.  Then, by letting $\delta \to 0$, it is shown there exists a sequence $\varphi_j \in C^{\infty}(\overline{M})$ with $\varphi_j \downarrow \varphi$ pointwise on $\overline{M}$.  To have smooth convergence uniformly on $V$ we modify this construction slightly as follows.

First we choose some finite open cover $\{V_{\alpha}\}_{\alpha=1}^{k}$ of the compact set $\overline{M}\backslash U$.  Moreover, by compactness we may assume for each $\alpha>0$ we have a proper inclusion of open sets $V_{\alpha} \subset U_{\alpha}  \subset W_{\alpha}$ where: $W_{\alpha}\bigcap \overline{V}=\emptyset$ and $W_{\alpha}$ is a holomorphic coordinate neighborhood with smooth local \K potential $\eta=i\partial\bar{\partial}f_{\alpha}$.  Then letting $U_0:=U$ we take $\{U_{\alpha}\}_{\alpha=0}^{k}$ as our open cover of $\overline{M}$.

 Next we define local approximations  $\varphi_{\alpha,\delta}$ of $\varphi$ on each $U_{\alpha}$.  For $\alpha=0$, let
$g_{\alpha}$ be a smooth function which is equal to $0$ in $U\backslash\cup_{\alpha\neq0}V_{\alpha}$
and equals $-1$ outside some compact subset of $U$.  For all $\alpha\neq0$, let $g_{\alpha}$ be a smooth function in
$U_{\alpha}$ such that $g_{\alpha}=0$ in $V_{\alpha}$ and $g_{\alpha}=-1$
outside some compact subset of $U_{\alpha}$. Assume that
$i\partial\bar{\partial}g_{\alpha}\geq-C\omega$ for some $C$ independent
of $\alpha$.  Now we define $\varphi_{\alpha,\delta}$ on $U_{\alpha}$ as follows. If
$\alpha\neq0$, then as in \cite{Blocki-Koloziej} we let $$\varphi_{\alpha,\delta}=u_{\alpha,\delta}-f_{\alpha}+\frac{\epsilon}{C}g_{\alpha}$$
where $u_{\alpha,\delta}$
is a mollification of $u_{\alpha}:=\varphi+f_{\alpha}\in Psh(U_{\alpha})$ in $W_{\alpha}$ so that $u_{\alpha,\delta} \downarrow u_{\alpha}$ as $\delta \to 0$. If $\alpha=0$, then we let  $$\varphi_{0,\delta}=\varphi+\frac{\epsilon}{C}g_{0}$$
In both cases, we have $\varphi_{\alpha,\delta}\in Psh(U_{\alpha},(1+\epsilon)\eta)$
and it is non-increasing as $\delta$ decreases.

Now given our open cover $\{U_{\alpha}\}_{\alpha=0}^{k}$, and local approximations $\varphi_{\alpha,\delta}$ in these, the proof of the claim follows exactly as in  \cite{Blocki-Koloziej} involving the regularized maximum (from \cite{Demailly}) of the $\varphi_{\alpha,\delta}$'s where the maximum is taken over all $\alpha$.  Then, by letting $\delta \to 0$.  We refer to \cite{Blocki-Koloziej}  for details of this argument.

\end{proof}

\section{Proof of  Theorem \ref{T-bdd data}}

Analogous to \cite{GZ}, we will use Theorem \ref{BK} to construct approximation solutions to K\"ahler-Ricci flow from which we will take a limit to obtain the desired solution.  Our study here is similar in ways to \cite{LZ} where the authors also studied cusp type K\"ahler-Ricci flow solutions.  On the other hand, our work here is different from these in the following ways.  First, our initial metric is not cusplike (or even complete) as in the case of \cite{LZ}, yet we are looking solutions which are cusplike for positive times.  This is one reason that the Monge \A equation \eqref{pma0} we consider is somewhat different than the ones studied in other works.  Second, the background Carlson-Griffith metric has only $L^{1}$ volume form so we cannot approximate $\phi_0$ using the procedure from \cite{ST} which is based on Kolodziej's $L^{p}$ estimate. In the following we will introduce  approximation procedures to overcome these difficulties.  
We first make the following technical assumptions which we will use throughout the rest of the section.

\begin{ass}\label{defsmoothcase}
 Let $\eta,  \widehat{\omega}$ and $\phi_0$ be as in Theorem \ref{T-bdd data} and $T_{[\omega_0 ]}$ be as in \eqref{existenceestimate1}, and assume $\|S\|_{\widehat{h}}^{2}<1$ on $\overline{M}$ where $S, \widehat{h}$ are used to define $\widehat{\omega}$.

 Fix some $\widetilde{T}<T_{[\omega_{0}]}$, then choose some $\tilde{c}>0$ and a smooth volume form $\widetilde{\Omega}$ on $\overline{M}$ so that for $\widetilde{h}:=\tilde{c}\widehat{h}$ we have:
\begin{enumerate}
\item $\eta+t(-Ric(\widetilde{\Omega})+\Theta_{\widetilde{h}})>0$ on  $\overline{M}\times [0, \widetilde{T}]$
\item$\eta+t(-Ric(\widetilde{\Omega})+\Theta_{\widetilde{h}}+\frac{2\Theta_{\widetilde{h}}} { \log\|S\|_{\widetilde{h}}^{2} })>0$ on $M\times [0, \widetilde{T}]$

\end{enumerate}

 We will abbreviate $\|S\|_{\tilde{c}\widehat{h}}^{2}, \Theta_{\tilde{c}\widehat{h}}$ and $\widetilde{\Omega}$ simply by $\|S\|^{2}, \Theta$ and $\Omega$ repsectively.

\end{ass}

 By the definition of $T_{[\omega_0]}$ we can then choose $\widetilde{\Omega}$ such that (1) holds for $t=\widetilde{T}$.  Then inequality in (2) will also hold at $t=\widetilde{T}$ for $\tilde{c}$ sufficiently small depending on $\widetilde{T}$ (by the smoothness of $\Theta$ on $\overline{M}$).  The fact that (1) and (2) holds for all $t\in [0, \widetilde{T}]$ then follows by interpolation between $t=0$ and $t=\widetilde{T}$.

\subsection{Proof of  Theorem \ref{T-bdd data} when $\varphi_0\in C^{\infty} (\overline{M})$}

From the results in \S 2.1 and also Theorem 8.19 in \cite{LZ} we have

\begin{lem}\label{LZ1.1}
For $\epsilon>0$ sufficiently small (depending only on $\omega_0$), $\omega_{\epsilon,0}=\omega_0-\e\ddb\log\log^{2}\|S\|_{\widehat{h}}^{2}$ is a Carlson-Griffiths metric on $M$  and \eqref{ckrf} has a bounded curvature solution $\omega_{\epsilon} (t)$ on $M\times[0, T_{[\omega_0 ]})$ with initial data $\omega_{\epsilon,0}$.
\end{lem}
  For any compact subset of $K\subset M$ we will derive estimates for $\omega_{\epsilon} (t)$ above on  $K\times [0, \widetilde{T}]$ which are uniform with respect to $\epsilon$ sufficiently small (depending on $\widetilde{T}$).  As $\widetilde{T}<T_{[\omega_0 ]}$ was arbitrary, this will ensure that $\omega_{\e}(t)$ converges locally unifomly uniformly on compact sets to a  smooth solution $\omega (t)$ to \eqref{ckrf} on $M\times[0, T_{[\omega_0 ]})$ with initial data $\omega_{0}=\eta+i\ddbar \varphi_{0}$.  

For any solution $\omega_{\epsilon}(t)$ we may write

\begin{equation}\label{999}\omega_{\epsilon} (t) =\theta_{\epsilon,t}+i\ddbar\varphi_{\epsilon}(t)\end{equation}
where $\varphi_{\epsilon}(t)$ solves the parabolic Monge \A equation on $M\times[0, T_{[\omega_0 ]})$:

\begin{equation}\label{pma-smooth-appr}
\begin{split}
\left\{
   \begin{array}{ll}
\partial_{t}\varphi_{\epsilon}(t)  = \displaystyle \log\frac{\|S\|^{2}\log\|S\|^{2}(\theta_{\epsilon, t}+i\partial\bar{\partial}\varphi_{\epsilon}(t))^{n}}{\Omega};\\
\varphi_{\epsilon}(0) =  \varphi_{0}.
  \end{array}
 \right.\\
\end{split}
\end{equation}

$$\,\,\,\,\,\,\,\,\, \theta_{\e,t}:=\eta+t\chi-\e \partial\bar{\partial}\log\log^{2}\|S\|_{\widehat{h}}^{2}; \hspace{12pt} \chi:=-\Ric(\Omega)+\Theta-i\partial\bar{\partial}\log\log^{2}\|S\|^{2}$$
\begin{rem} Note that we are using a different Hermitan metric in the last terms in $\theta_{\e,t}$ and $\chi$.  This is to ensure that $\omega_{0, \e}$, and thus the solution $\omega_{\e} (t)$ on on $M\times[0, T_{[\omega_0 ]})$ is independent of $\widetilde{T}$.
\end{rem}

Indeed, letting $$\varphi_{\e}(t) = \varphi_0+\int_0^t  \displaystyle \log\frac{\|S\|^{2}\log^2\|S\|^{2}(\omega_{\e}(t))^{n}}{\Omega}$$ and defining $\theta_{\e, t}$ as above we see that \eqref{pma-smooth-appr} is obviously satisfied.  On the other hand,  we have $\theta_{\epsilon,0}+i\ddbar\varphi_{\epsilon}(0)=\omega_{\e}(0)$ while on $M\times[0, T_{[\omega_0 ]})$ 
\begin{equation}
\begin{split}
(\theta_{\e, t}+i\ddbar  \varphi_{\e}(t))'=&-\Ric(\Omega)+\Theta-i\partial\bar{\partial}\log\log^{2}\|S\|^{2}\\
&+i\ddbar ( \log\|S\|^{2} + \log\log^2\|S\|^{2}+\log(\omega^n_{\e}(t))-\log \Omega)\\
=& i\ddbar \log(\omega^n_{\e}(t))\\
=&-Rc(\omega_{\e}(t))
\end{split}
\end{equation}
and it follows from \eqref{ckrf} that \eqref{999} holds on $M\times[0, T_{[\omega_0 ]})$.  Conversely, reversing the process above shows that the metric in \eqref{999} satisfies \eqref{ckrf} for any given solution to \eqref{pma-smooth-appr}.

 Now for all $\epsilon$ sufficiently small (depending on $\widetilde{T}$), we will derive estimates for $\varphi_{\epsilon}(t)$ on $M\times(0, \widetilde{T}]$ which may depend on $\widetilde{T}$ but will be independent of $\epsilon$, yielding corresponding $C_{loc}^{\infty}$ estimates for  $\omega_{\epsilon}(t)$.  These will allow us to let $\epsilon \to 0$ and $\widetilde{T}\to T_{[\omega_0 ]}$ and obtain a limit solution on $M\times[0, T_{[\omega_0 ]})$ satisfying the conditions in Theorem \ref{T-bdd data}.  

 Note by (2) in Assumption \ref{defsmoothcase}  and \eqref{CGmetric}, for $\epsilon>0$ sufficeintly small (depending on $\widetilde{T}$) we may have
\begin{equation}\label{t-bdde1}c_{1}(t+\e)\widehat{\omega}\leq\theta_{t,\e}\leq c_{2}\widehat{\omega},\end{equation}
  for some constants $c_1, c_2>0$ independent of $\e$ for all $t\in [0, \widetilde{T}]$.

We begin with the following $C^{0}$-estimates:
\begin{lem}\label{C0-smooth-appr}
On the time interval $(0,\widetilde{T}],$ we have $|\varphi_{\e}|\leq C,\;|\dot{\varphi_{\e}}|\leq\frac{C}{t},$ where $C$ is independent of $\e$ sufficiently small.
\end{lem}
\begin{proof}
Fix $\epsilon>0$ such that \eqref{t-bdde1} holds.    By the results in \S 2.1,  $\theta_{t,\e}$ has bounded curvature on $M=\overline{M}\setminus D,$ uniformly for all $t\in [0, \widetilde{T}]$.   Let $\psi:=\varphi_{\e} - Ct$ for some $C>0$ to be chosen.    Then since $\omega_{\e}(t)$ is a bounded curvature solution, $\psi$ is uniformly bounded on $M\times[0, \widetilde{T}]$.   We first suppose that $\psi$ attains a maximum value on $M\times[0,\widetilde{T}]$ at some point $(\overline{x},\overline{t})$.  If $\bar{t}>0$ then using \eqref{pma-smooth-appr} we have at $(\overline{x},\overline{t})$

\begin{equation}\label{e1.1}
\partial_{t}\psi_{\e}  \leq  \log\frac{\|S\|^{2}\log\|S\|^{2}\theta_{\bar{t}, \e}^{n}}{\Omega}-C<-1\\
\end{equation}
 by choosing $C$ sufficiently large independent of $\e$. Indeed, such a choice of $C$ exists from the upper bound in \eqref{t-bdde1} and property (3) in \S 2.1.  But \eqref{e1.1} contradicts the maximality of $\psi$ unless $\bar{t}=0$ in which case $\psi \leq \sup \psi(0)=\sup\varphi_{0} $ on $M\times [0, \widetilde{T}]$ and thus

 \begin{equation}\label{666}\varphi_{\e}\leq\sup\varphi_{0}+C\widetilde{T}\end{equation} on $M\times [0, \widetilde{T}]$.  Now in general, suppose $\psi$ does not attain a maximal value on $M\times [0, \widetilde{T}]$.
Since $\omega_{\e}(t)$ is a bounded curvature solution, we have $|\partial_{t}\psi_{\e}|$ is bounded on $M\times[0,\widetilde{T}]$, and by the Omori-Yau maximum principal we can find a sequence $x_{k}\in M$ and $\bar{t}\in[0, \widetilde{T}]$  such that $\psi_{\epsilon}(x_{k},\bar{t})\to\sup_{M\times[0,\widetilde{T}]}\psi_{\epsilon}$
and $i\partial\bar{\partial}\psi_{\epsilon}(x_{k},\bar{t})\leq\lambda_{k}\theta_{\bar{t}, \e}$
where $\lambda_{k}$ decreases to $0$ as $k\to\infty$.    Combining these with \eqref{pma-smooth-appr}, we may argue as above that for any $\delta$ we have $\partial_{t}\psi_{\e}(x_{k},\bar{t})  \leq -1$ for some $C$ independent of $\e$ and  all $k$ sufficeintly large.   On the other hand, $|\partial_{t}^{2}\psi_{\e}|(x_{k},\bar{t}) $ is uniformly bounded independent of $k$ using again that $\omega_{\e}(t)$ is a bounded curvature solution.  These last two facts contradict that $\psi_{\epsilon}(x_{k},\bar{t})\to\sup_{M\times[0,\widetilde{T}]}\psi_{\epsilon}$ unless $\bar{t}=0$ and we conclude again as above that \eqref{666} likewise holds on $M\times[0,\widetilde{T}]$ in this case as well.

Using again that $\omega_{\e}(t)$ is a bounded curvature solution, we also have $|\partial_{t}^{2}\psi_{\e}|$ is bounded on $M\times[0,\widetilde{T}]$. 

 By a similar argument, using \eqref{t-bdde1}  and \eqref{pma-smooth-appr} we may also get
\begin{equation}\label{777}\inf\varphi_{0}+C'\int_{0}^{t}\log(s+\e)ds\leq \varphi_{\e}\end{equation}
on $M\times [0, \widetilde{T}]$ for some constant $C' >0$ independent of $\e$.   We thus conclude the estimates for $|\varphi_{e}|$ in the Lemma hold.

 Next we will apply the methods in \cite{ST} to derive estimates for $\dot{\varphi_{\e}}$. We have
$$(\p_{t}-\Delta)\dot{\varphi_{e}}=tr_{\omega_{\e}}\chi$$ where the $\Delta$ denotes the Laplacian with respect to $\omega_{\e},$ and also $$(\p_{t}-\Delta)(t\dot{\varphi_{e}}-\varphi_{\e}-nt)=-tr_{\omega_{\e}}(\eta-\e\ddb\log\log^{2}\|S\|^{2})<0,$$ where
the last inequality comes from the fact that $\eta-\e\ddb\log\log^{2}\|S\|_{h}^{2}>0$ when $\e$ is small by  \eqref{CGmetric}.
Applying the maximum principle in \cite{Shi2}, we conclude  the supremum of $(t\dot{\varphi_{e}}-\varphi_{\e}-nt)$ on $M\times(0,\widetilde{T}]$, which is indeed finite, is attained when $t=0$ and thus $\dot{\varphi_{\e}}\leq\frac{\varphi_{\e}-\varphi_{0}}{t}+n$ on $M\times(0,\widetilde{T}]$.  On the other hand, for sufficiently large $A$ independent of $\e$ we have the following in which $c$ and $C$ refer to constants which may differ from line to line, may depend on $\widetilde{T}$ and $A$, but are independent of $\e$:
\begin{equation}\label{222}
\begin{split}
&(\p_{t}-\Delta)(\dot{\varphi_{e}}+A\varphi_{\e}-n\log t)\\
=&\;tr_{\omega_{\e}}(\chi+A\theta_{t,\e})+A\log\frac{\omega_{\e}^{n}\|S\|^{2}\log^{2}\|S\|^{2}}{\Omega}-(An+\frac{n}{t})\\
=&\; A tr_{\omega_{\e}}(\theta_{t+1/A,\e})+A\log\frac{\omega_{\e}^{n}\|S\|^{2}\log^{2}\|S\|^{2}}{\Omega}-(An+\frac{n}{t})\\
\geq&\; A tr_{\omega_{\e}}(c \widehat{\omega})+A\log\frac{\omega_{\e}^{n}}{\widehat{\omega}^{n}}-\frac{C}{t}\\
\geq&\; c tr_{\omega_{\e}}\widehat{\omega}+A\log\frac{\omega_{\e}^{n}}{\widehat{\omega}}-\frac{C}{t}\\
\geq&\;\frac{c}{2}\left(\frac{\widehat{\omega}^{n}}{\omega_{\e}^{n}}\right)^{1/n}-\frac{C}{t}\\.
\end{split}
\end{equation}
 where in the fourth line we have made use of \eqref{t-bdde1}, property (3) in \S2.1 and the equivalence of any two Carlson-Grifiths metrics on $M$, and in the sixth line we have again used \eqref{t-bdde1}.   Now let $\psi=(\dot{\varphi_{e}}+A\varphi_{\e}-n\log t)$,
and assume $\psi$ attains a minimum value on $M\times[0,\widetilde{T}]$ at some point $(\overline{x},\overline{t})$.  It follows that $\bar{t}>0$, and \eqref{222} then gives $\omega_{\e}^{n}(\overline{x},\overline{t}) \geq c\widehat{\omega}^{n}(\overline{x},\overline{t})\bar{t}^{n}$ for some $c>0$ independent of $\e$.  From this, \eqref{t-bdde1}, \eqref{pma-smooth-appr} and property (3) in \S 2.1 we may then have $\psi \geq C\log t-C$ and thus $\dot{\varphi_{e}}\geq C\log t-C$ on $M\times(0,\widetilde{T}]$ for some $C>0$ where we have used \eqref{666} and \eqref{777}.  In general, $\psi$ is only bounded but may not attain a minumum value on $M\times[0,\widetilde{T}]$, though we may argue as before, applying the above estimate along an appropriate space-time sequence obtained by the Omori-Yau maximum principle, to conclude $\dot{\varphi_{e}}\geq C\log t-C$ on $M\times(0,\widetilde{T}]$ for some $C>0$ in this case as well.  We thus conclude the estimates for $|\dot{\varphi_{e}}|$ in the Lemma hold.
\end{proof}

Next we want to derive a Laplacian estimate for $\varphi_{e}.$
\begin{lem}\label{C2-smmoth-appr} for each $t\in (0,\widetilde{T}]$ we have
$C_{1}(t)\widehat{\omega}\leq\omega_{\e}(t)\leq C_{2}(t)\widehat{\omega},$ for  constants $C_{1}(t),C_{2}(t)>0$ independent of $\epsilon$ sufficiently small.
\end{lem}
\begin{proof}
First recall that $\widehat{\omega}$ is complete and has uniformly bounded bisectional curvature. Then the parabolic version
of the Chern-Lu inequality (see \cite{ST0}) gives:$$(\p_{t}-\Delta)\log\;tr_{\omega_{\e}}\widehat{\omega}\leq C tr_{\omega_{\e}}\widehat{\omega}+C$$ where the constant $C$ depends on the
upper bound of the bisectional curvature of $\widehat{\omega}.$ For the rest of the proof, $C$ will denote a constant, which may change from line to line, and which is independent of $\e$.   Now by \eqref{t-bdde1} we may choose $A$  sufficiently large independent of $\e$ so that
\begin{equation}\label{333}
\begin{split}
&(\p_{t}-\Delta)(t\log\;tr_{\omega_{\e}}\widehat{\omega}-A\varphi_{e})\\
\leq&\;tr_{\omega_{\e}}(Ct\widehat{\omega}-A\theta_{t,\e})+\log\;tr_{\omega_{\e}}\widehat{\omega}-A\log\frac{\omega_{\e}^{n}\|S\|^{2}\log^{2}\|S\|^{2}}{\Omega}
+(An+Ct)\\ \leq&\;-\frac{A}{2}tr_{\omega_{\e}}\theta_{t,\e}+\log\;tr_{\omega_{\e}}\theta_{t,\e}+A\log\frac{\theta_{t,\e}^{n}}{\omega_{\e}^{n}}-CA\log(t+\e)+C\\
\leq&\;-\frac{A}{4}t
r_{\omega_{\e}}\theta_{t,\e}-CA\log(t+\e)+C,
\end{split}
\end{equation}
Now suppose $t\log\;tr_{\omega_{\e}}\widehat{\omega}-A\varphi_{e}$ attains a maximum value on $M\times[0,\widetilde{T}]$ at some point $(\overline{x},\overline{t})$.  Then if $\overline{t}>0$, using \eqref{333} we have at $(\overline{x},\overline{t})$ that
$tr_{\omega_{\e}}\theta_{t,\e}\leq-C\log t+C$ and so $$t\log\;tr_{\omega_{\e}}\widehat{\omega}-A\varphi_{e}\leq Ct(\log\log\frac{1}{t}+\log\frac{1}{t+\e})+C\leq C.$$  In this case it follows that \begin{equation}\label{444}tr_{\omega_{\e}}\widehat{\omega}\leq e^{\frac{C}{t}}\end{equation} on $M\times[0,\widetilde{T}]$ where we have used estimate for $|\varphi_{\e}|$ Lemma \ref{C0-smooth-appr}.  On the other hand, if $\overline{t}=0$ we also clearly have \eqref{444}.  In general, In general, $t\log\;tr_{\omega_{\e}}\widehat{\omega}-A\varphi_{e}$ is only bounded but may not attain a maximum value on $M\times[0,\widetilde{T}]$, though we may argue as in the proof of Lemma \ref{C0-smooth-appr}
and apply the above estimates along an appropriate sequence in space-time to likewise conclude that \eqref{444} holds on $M\times[0,\widetilde{T}]$ in this case as well.

Finally, for any $t\in (0,\widetilde{T}]$,
\eqref{pma-smooth-appr} gives
\begin{equation}\label{555}\frac{\omega_{\e}^{n}}{\widehat{\omega}^{n}}=e^{\dot{\varphi_{\e}}}\frac{\Omega}{\widehat{\omega}^{n}\|S\|^{2}\log^{2}\|S\|^{2}}\leq C_t\end{equation} on $M$ for some constant $C_t>0$ where we have used the estimate for $|\dot{\varphi_{\e}}|$ in Lemma \ref{C0-smooth-appr} and (3) in \S 2.1.  The Lemma follows from \eqref{555} and \eqref{444}.

\end{proof}

\begin{proof}[Completion proof of Theorem \ref{T-bdd data}  when $\varphi_0\in C^{\infty} (\overline{M})$]
The previous two lemmas and the Evans-Krylov theory applied to \eqref{pma-smooth-appr} imply that for any $K\subset \subset M$ and $s\in(0, \widetilde{T})$ we have $$\max_K \| \nabla_{\eta}^k \varphi_{\e}(t) \|_{\eta} \leq C_{k,s, K,t}$$ independent of $\e$ (sufficiently small depending on $\widetilde{T}$) and $t\in(s, \widetilde{T}]$ where the norm and covariant derivative here are with respect to $\eta$.  Thus for some subsequence  $\e_{i} \to 0$, $\varphi_{\e_{i}}$ will converge locally uniformly to a smooth solution $\varphi$ to \eqref{pma0} on $M\times(0, \widetilde{T})$ which is bounded on $M$ for each $t$.   Thus $\omega_{\epsilon_i}(t)$ converges locally uniformly to a smooth solution $\omega(t)$ to the flow in equation \eqref{ckrf} on $M\times(0, \widetilde{T})$ and as $\widetilde{T}<T_{[\omega_0 ]}$ was arbitrary, we may in fact assume the convergence to a solution on $M\times(0, T_{[\omega_0 ]})$ satisfying the estimates in Lemma \ref{C2-smmoth-appr}.

We now show the limit solution $\omega(t)$ in fact converges smoothly uniformly on compact subsets of $M$ to $\omega_0$ as $t\to 0$, and thus can be extended to a smooth solution to \eqref{ckrf} on  $M\times[0, T_{[\omega_0 ]})$ with initial data $\omega_0$.  This basically follows from applying Theorem 4.1 in \cite{Chau-Li-Tam} to the sequence $\omega_{\epsilon_i}(t)$, and observing that the completeness of the background metric $\hat{g}$ in that Theorem is in fact not necessary in our case.  We describe this in more detail as follows.  Consider the family of solutions  $\omega_{\epsilon_i}(t)$ on $M\times[0, T_{[\omega_0 ]})$.   For each $i$ we have $\omega_{\epsilon_i}(0) \geq c\eta$ on $M$ for some $c>0$ independent of $i$ and we conclude from the proof of Lemma 3.1 of \cite{Chau-Li-Tam} that \begin{equation}\label{1}\omega_{\epsilon_i}(t) \geq c\eta\end{equation} on $M\times[0, T]$ for some $c, T>0$ independent of $i$.  For this simpy observe that in Lemma 3.1 of \cite{Chau-Li-Tam}, the completeness of $\hat{g}$ is never actually used in the proof.   Next we choose any smooth non-negative function $\psi:[0, \infty)\to \R$ which is identically zero in some neighborhood of $0$ and let $\varphi(t):=\psi(\|S\|^2)$.  Then using \eqref{1}, as in the  proof of Lemma 3.3 of \cite{Chau-Li-Tam} we may have $\|\nabla_{\epsilon_i, t} \varphi(t)\|, \|\Delta_{\epsilon_i, t} \varphi(t)\|\leq C$ on $M$ where the norms here are relative to $\omega_{\epsilon_i}(t)$ and $C$ is independent of $i$ and $t\in [0, T)$, and by the same proof there we may conclude that $\varphi(t)R_{\epsilon_i} (t)\geq C$ on $M\times[0, T]$ where $R_{\epsilon}(t)$ is the scalar curvature of $\omega_{\epsilon_i}(t)$ and the constant $C$ is independent of $i$.  From this and the fact that $\phi$ was arbitrarily chosen, we can conclude as in  Lemma 3.3  of \cite{Chau-Li-Tam}  that for any compact $K\subset \subset M$ we have the upper bound
\begin{equation}\label{2}\omega_{\epsilon_i}(t) \leq c\eta\end{equation} on $K\times[0, T]$ for some $c$ independent of $i$.   The smooth convergence of  $\omega_{\epsilon_i}(t)$  on compact subsets of $M\times[0, T]$ then follows from \eqref{1}, \eqref{2} and the Evans-Krylov theory.  Thus the limit solution $\omega(t)$ extends smoothly to $M\times[0, T_{[\omega_0 ]})$ as claimed above.

To complete the proof of the Theorem in this case, it remains only to prove the uniqueness statement which we do in the next sub-section in Proposition \ref{pma-smooth-unique}.
\end{proof}

\subsection{Proof of  Theorem \ref{T-bdd data} when \\$\varphi_0\in L^{\infty}(M)\bigcap C^{\infty} (M) \bigcap Psh(\overline{M}, \eta)$}

By \eqref{lower} we can choose some $\e>0$ so that in fact we have $\varphi_0\in Psh(\overline{M}, (1-\e)\eta)$.  Then, using Theorem \ref{BK} we may choose a sequence $\{\varphi_{j}\} \subset C^{\infty} (\overline{M})\bigcap  Psh(\overline{M}, (1-\e)\eta)$ so that $\varphi_{j}\downarrow\varphi_{0}$ pointwise on $\overline{M}$ and locally smoothly on $M$.  In particular, it follows that
\begin{enumerate}
\item $|\varphi_{j}|\leq C$ for all $j$ and some $C$
\item $\omega_{j}=\eta+i\ddbar \varphi_j\geq \e\eta$  for all $j$ on $\overline{M}$.
\end{enumerate}

Now for each $j$ we let $\omega_j(t)$ be the solution to \eqref{ckrf} on $M\times[0, T_{[\omega_0 ]})$ with initial data $\omega_j(0)= \eta +i\ddbar \varphi_j$ constructed in \S 3.1.  We may then apply the estimates in \S 3.1 to each solution $\omega_j(t)$, and observe that the estimates are in fact independent $j$, thus providing a subsequential limit solution $\omega(t)$ satisfying the conclusions of Theorem \ref{T-bdd data}.  We explain this in more detail as follows.

Under Assumption 1,  let $\varphi_j(t)$ be the corresponding solution to \eqref{pma0} on $M\times[0,T_{[\omega_0 ]})$ also as in \S 3.1.  From Lemma \ref{C0-smooth-appr} and (1) it follows that $|\varphi_{j}|(t)$ and $|\dot{\varphi_{j}}(t)|$ are uniformly bounded on compact subsets of $M\times (0, \widetilde{T})$ independently of $j$ where $\widetilde{T}$ is from Assumption 1.   From Lemma \ref{C2-smmoth-appr} it further follows that $\omega_j(t)$ is uniformly equivalent to $\widehat{\omega}$ on $M$, independent of $j$, on compact intervals of $(0,\widetilde{T})$.   As $\widetilde{T}<T_{[\omega_0 ]}$ was arbitrary and by arguing as in \S 3.1 separately for each $j$, we may conclude smooth local  estimates for $\omega_j(t)$ on compact subsets of $M\times[0, T_{[\omega_0 ]})$ which are independent of $j$, and that some subsequence of $\omega_j(t)$ converges to a solution $\omega(t)$ to \eqref{ckrf} on $M\times[0, T_{[\omega_0 ]})$ satisfying \eqref{completeness} for all $t\in (0, T_{[\omega_0 ]})$.

The proof of Theorem \ref{T-bdd data} will be complete once we prove uniqueness of such a solution which we do in the proposition below.  First, we note that by \eqref{lower} each $\omega_j(t)$ will satisfy the lower bound in \eqref{1} for constants $c, T>0$ independent of $j$ and thus the limit solution $\omega(t)$ likewise satisfies the lower bound in \eqref{1}.  We will use this fact in the following proof.

\begin{prop}\label{pma-smooth-unique}
Let $\varphi_{1}(t),\varphi_{2}(t)$ be two solutions to \eqref{pma0} on $M\times[0,T_{[\omega_0 ]})$ with initial data $\varphi_{1}(0)=\varphi_{2}(0)\in L^{\infty}(M)\bigcap C^{\infty} (M)$.  Suppose   $|\varphi_{1}(t)|,|\varphi_{2}(t)|$ are both bounded on $M\times[0,T)$ for every $T<  T_{[\omega_0 ]}$.  Then $\varphi_{1}(t)=\varphi_{2}(t)$ on $M\times[0,T]$.
\end{prop}
\begin{proof}

 We assume without loss of generality that the solutions $\omega_1(t)$ and $\varphi_1(t)$ are as constructed as in the proof of Theorem \ref{T-bdd data} so far.  Now for any $T<T_{[\omega_0 ]}$ we prove that $\varphi_2\leq \varphi_1$ on $M\times[0,T]$.  Let $|\cdot|^2$ be a Hermitian metric such that $\|S\|^2<1$ and let $\Theta$ denotes its curvature form.  As noted above, $\omega_1(t)$ satisfies the inequality in \eqref{1} on $M\times[0,\e]$ for some constant $c, \e >0$.  Thus for all $a>0$ we can find $C_a\to 0$ as $a\to 0$ such that $\log\frac{(\omega_1(t)+a\Theta)^n}{\omega_1(t)^n}< C_a$ on $M\times [0,T]$.  Consider $\psi=\varphi_2-\varphi_1+a\log\|S\|^2-C_at$, since $\varphi_2-\varphi_1$ is a bounded function, $\psi$ attains a maximum on $M\times[0,T]$ at some point $(\bar{x},\bar{t})$. If $\bar{t}>0$, then at $(\bar{x},\bar{t})$ we have
\begin{align*}
0&<\partial_t \psi=\log \frac{(\omega_1(t)+a\Theta+i\ddbar\psi)^n}{\omega_1(t)^n}-C_a<0,
\end{align*}
a contradiction. Thus since $\psi(x,0)<0$ we have $\varphi_2\leq\varphi_1-a\log\|S\|^2+C_at$, and letting $a\to 0$ we get $\varphi_2\leq \varphi_1$.

To prove $\varphi_2\geq \varphi_1$ we argue similarly.  Namely, we first choose $C_a \to 0$ as $a\to 0$ so that  $\log(\omega_1-a\Theta)^n/\omega_1^n \geq C_a$.  Then we let $\psi=\varphi_2-\varphi_1-a\log\|S\|^2-C_1t$ and argue as before using the maximum principle that $\psi\geq \min_M \psi(0)$ everywhere then conclude by letting $a \to 0$ that $\varphi_2 \geq \varphi_1$ on $M$.

\end{proof}

\section{Proof of Theorem \ref{T1}}

The proof here roughly the same steps as the proof in Theorem \ref{T-bdd data}.   Namely, we construct a suitable approximating family for $\omega_0$, then consider the corresponding family of approximate solutions to \eqref{ckrf} and convergence to a limit solution is proved using the parabolic Monge \A equation.  One major difference here is that $\omega_0$ is complete on $M$ and we want to preserve this property in our approximation.  Another major difference is that $\phi_0$ is no longer bounded on $M$ in general which again makes the estimates more difficult.

 We make the following technical assumptions which we will use throughout the rest of the section.

\begin{ass}\label{ass2}
Let $\eta, \widehat{\omega}, \phi_0,$ be as in Theorem \ref{T1}, and assume $\|S\|_{\widehat{h}}^{2}<1$ on $\overline{M}$ where $S, \widehat{h}$ are used to define $\widehat{\omega}$.

 Fix some $\widetilde{T}<T_{[\omega_{0}]}$, then choose some $\tilde{c}, \widetilde{\beta}>0$ and a smooth volume form $\widetilde{\Omega}$ on $\overline{M}$ so that for $\widetilde{h}:=\tilde{c}\widehat{h}$: 

\begin{enumerate}
\item  on $M\times [0, \widetilde{T}]$ we have
$$\eta+t(-\Ric(\widetilde{\Omega})+\Theta_{\widetilde{h}}+\frac{2\Theta_{\widetilde{h}}}{\log\|S\|_{\widetilde{h}}^{2}})>0$$
\item on $\overline{M}$ we have
$$(1-c_{\widetilde{\beta}})\eta\leq\eta-\widetilde{\beta}\Theta_{{\widetilde{h}}} \leq(1+c_{\widetilde{\beta}})\eta$$
for some $c_{\widetilde{\beta}}<\frac{1}{2}$ with $T<(1-c_{\widetilde{\beta}})\widetilde{T}$.
\vspace{10pt}
\item $\log\log^{2}\|S\|_{\widetilde{h}}^{2}>1$ and $\eta+i\ddbar \log\log^2\|S\|_{\widetilde{h}}^2>0$  on $M$.
\end{enumerate}

\vspace{10pt}
As before we will abbreviate $\|S\|_{\widetilde{h}}^{2}, \Theta_{\widetilde{h}}$ and $\widetilde{\Omega}$ simply by $\|S\|^{2}, \Theta$ and $\Omega$.

\end{ass}
 As in the case of Assumption 1, we can find a smooth volume form $\widetilde{\Omega}$, then $\tilde{c}$ sufficiently small so that the inequality in (1) holds for $t=\widetilde{T}$ in which case it must also hold for all $t\in  [0, \widetilde{T}]$ by interpolation.   A choice of $\widetilde{\beta}$ in (2) is justified  by the smoothness of $\Theta$ on $\overline{M}$ and by choosing $\tilde{c}$ smaller still if necessary.  Finally, by scaling $\tilde{c}$ smaller again we may assume the inequality in (3) holds and that $\eta+i\ddbar \log\log^2\|S\|_{\widetilde{h}}^2$ is also a Carlson-Griffiths metric as in \S 2.1.

\subsection{Approximate solutions $\omega_{\alpha, j}(t)$}

Recall in Theorem \ref{T1} we have $\varphi_0\in C^{\infty}(M)\bigcap Psh(\overline{M}, \eta)$ with zero Lelong number such that \begin{equation}\label{splitoff0}\omega_0=\eta+i\ddbar\varphi_0\geq 2\delta\widehat{\omega}\end{equation} on $M$ for some $\delta>0$.  We begin by construct a two parameter family $\varphi_{\alpha,j}$ approximating $\varphi_0$ as $\alpha\to 0$ and $j\to \infty$ so that the metrics $\omega_{\alpha,j}(0)=\eta+i\ddbar\varphi_{\alpha,j}$ are likewise  bounded below for some fixed Carlson-Griffiths metric for all $\alpha$.  This uniform lower bound will be key for our later proofs, and this is one reason for our two parametrer construction as opposed to a single parametrer approximation as in Theorem \ref{BK}.

\begin{lem}\label{approxsequence} There exists $\hat{\alpha}$ such that for all $0<\alpha \leq \hat{\alpha}$ there exists a sequence $\varphi_{\alpha,j}\in Psh(\bar{M},\eta)$ such that
\begin{enumerate}
\item  $\varphi_{\alpha,j}$ decreases to $\alpha\log\|S\|_{\widehat{h}}^{2}+\varphi_0$ (as $j\to\infty$) pointwise and smoothly on compact subsets of $M$.
\item $\psi_{\alpha,j}=\varphi_{\alpha,j}+\delta\log\log^{2}\|S\|_{\widehat{h}}^{2} \in C^{\infty}(\bar{M})\cap Psh(\bar{M},(1-\delta)\eta)$
\end{enumerate}
\end{lem}
\begin{proof}
We will make use of the definitions and results of \S 2.1.  Now $i\ddbar \log\|S\|_{\widehat{h}}^{2}$ coincides with a smooth form on $\overline{M}$ and thus for $\alpha>0$ sufficiently small we have $-\delta\widehat{\omega} \leq \alpha i\ddbar \log\|S\|_{\widehat{h}}^{2}\leq \delta\widehat{\omega}$ on $M$ and it follows from \eqref{splitoff0} that

\begin{equation}\label{splitoff}(1-\delta)\eta+i\ddbar(\alpha\log\|S\|_{\widehat{h}}^{2}+\delta\log\log^{2}\|S\|_{\widehat{h}}^{2}+\varphi_0)\geq\delta\widehat{\omega}
\end{equation}
    In particular, since $\varphi_0$ has zero Lelong number (see definition \ref{psh}) the potential on the LHS of \eqref{splitoff} approaches $-\infty$ when approaching $D$ giving \begin{equation}\label{smallalpha2}\alpha\log\|S\|_{\widehat{h}}^{2}+\delta\log\log^{2}\|S\|_{\widehat{h}}^{2}+\varphi_0\in Psh(\bar{M},(1-\delta)\eta)\end{equation} for all sufficiently small $\alpha>0$.  Thus by Theorem \ref{BK} there exists $\psi_{\alpha,j}\in C^{\infty}(\bar{M})\cap Psh(\bar{M},(1-\delta)\eta)$
decreasing to $\alpha\log\|S\|_{\widehat{h}}^{2}+\delta\log\log^{2}\|S\|_{\widehat{h}}^{2}+\varphi_0$
as $j\to\infty$ pointwise on $M$ and smoothly on compact sets.   In particular, we have
\begin{equation}\label{splitoff3}
\eta+i\ddbar(-\delta\log\log^{2}\|S\|_{\widehat{h}}^{2}+\psi_{\alpha,j})= \delta\widehat{\omega}+(1-\delta)\eta+i\ddbar\psi_{\alpha,j}>\delta\widehat{\omega}
\end{equation}
so that $\varphi_{\alpha, j}:=-\delta\log\log^{2}\|S\|_{\widehat{h}}^{2}+\psi_{\alpha,j}\in Psh(\bar{M},\eta)$ decreases to \\$\alpha\log\|S\|_{\widehat{h}}^{2}+\varphi_0$
as $j\to\infty$ pointwise on $M$ and smoothly on compact sets.  Thus for all $\alpha>0$ sufficiently small $\varphi_{\alpha,j}$ satisfies the conclusions in the Lemma

\end{proof}

\begin{lem}\label{LZ}
For each $\varphi_{\alpha,j}$ in Lemma \ref{approxsequence}, $\omega_{\alpha,j}(0)=\eta+i\ddbar\varphi_{\alpha,j}$ is complete with bounded curvature on $M$ with lower bound
\begin{equation}\label{lowerbound}\omega_{\alpha,j}(0) \geq \delta \widehat{\omega}\end{equation}
and the \K Ricci flow \eqref{ckrf} has a smooth maximal bounded curvature solution $\omega_{\alpha,j}(t)$ on $M\times[0, T_{[\omega_0]})$ with initial condition $\omega_{\alpha,j}(0)=\eta+i\ddbar\varphi_{\alpha,j}$
  where $T_{[\omega_0]}$ is as in definition \ref{existenceestimate1}.
\end{lem}
\begin{proof}
 The lower bound in \eqref{lowerbound} follows immediately from \eqref{splitoff3}. On the other hand we can also write  $\eta+i\ddbar\varphi_{\alpha,j}=\eta+i\ddbar(-\delta\log\log^{2}\|S\|_{\widehat{h}}^{2}+\psi_{\alpha,j})$ where  $\psi_{\alpha,j}:=\varphi_{\alpha,j}+\delta\log\log^{2}\|S\|_{\widehat{h}}^{2} \in C^{\infty}(\overline{M})$ as in Lemma \ref{approxsequence}  (2), and the Lemma then follows from Theorem 8.19 in \cite{LZ} (see also example 8.15).

\end{proof}

 Consider $\omega_{\alpha, j}(t)$ as in the above Lemma.  In the following we will derive local estimates for $\omega_{\alpha,j}(t)$ on $M\times(0,\widetilde{T})$ which will be independent of $\alpha$ sufficiently small and  $j$ sufficiently large (depending on $\widetilde{T}$). These will ensure
 $\omega_{\alpha,j}(t)$ converges in $C_{loc}^{\infty}$
to a solution $\omega(t)$ on $M\times(0,\widetilde{T})$ as $\alpha\to 0$ and
$j\to\infty$. Then we will show that $\omega_{\alpha,j}(t)$ in fact
converges in $C_{loc}^{\infty}$ on $M\times[0,\widetilde{T})$ and the limit
solution is complete for a short time. Since $\widetilde{T}<T_{[\omega_0]}$ in Assumption 2 was arbitrary,
a diagonal argument will provide a solution on $M\times[0,T_{[\omega_0]})$ as in Theorem \ref{T1}.  We may derive as in \eqref{pma-smooth-appr} that
\begin{equation}\label{e0aj} \omega_{\alpha,j}(t)=\theta_t+i\ddbar\varphi_{\alpha,j}(t)\end{equation}
where $\varphi_{\alpha,j}(t) $ solves the parabolic Monge \A equation:

\begin{equation}\label{pma}
\begin{split}
\left\{
   \begin{array}{ll}
\partial_{t}\varphi_{\alpha,j}(t)  = \displaystyle \log\frac{\|S\|^2\log\|S\|^2(\theta_{t}+i\partial\bar{\partial}\varphi_{\alpha,j}(t))^{n}}{\Omega};\\
\varphi_{\alpha,j}(0) =  \varphi_{\alpha,j}.
  \end{array}
 \right.\\
\vspace{10pt}
\theta_{t}:=\eta+t\chi; \hspace{12pt} \chi:=-\Ric\Omega+\Theta-i\partial\bar{\partial}\log\log^{2}\|S\|^{2}\\
\end{split}
\end{equation}
on $M\times[0, T_{[\omega_0 ]})$.

\begin{rem} Note that $\|S\|^2= \|S\|_{\widetilde{h}}^2$ is used in  \eqref{pma} while  $\|S\|_{\widehat{h}}^2$ is used in Lemma \ref{approxsequence} in defining $ \varphi_{\alpha,j}$.  This is to ensure that the solutions $\omega_{\alpha,j}(t)$ on $M\times[0, T_{[\omega_0]})$ are independent of $\widetilde{T}$.
\end{rem}

We will derive local uniform bounds of $|\varphi_{\alpha,j}(t)|$ and $|\partial_{t}\varphi_{\alpha,j}(t)|$
on compact subsets of $M\times(0,\widetilde{T}]$  where the bounds will be independent of $\alpha, j$.  We will use these to derive uniform local trace estimates for $\omega_{\alpha,j}(t)$ which combined with the local Evans-Krylov estimates will yield the desired uniform $C_{loc}^{\infty}$ estimates.

\subsection{A priori estimates for $\omega_{\alpha, j}(t)$}
Recall the choices for $0<T<\widetilde{T}<T_{[\omega_0 ]}$ and $\widetilde{h}, \widetilde{\Omega}, \widetilde{\beta}$ in Assumption \ref{ass2} and the notation there. Recall also the definitions of $\theta_t$ and $\chi$ from \eqref{pma}.    We fix some $\hat{\alpha}$ from Lemma 4.1 and will always assume that $\alpha \leq \hat{\alpha}$ in the following.

\subsubsection{Local $C^{0}$ estimates of $\varphi_{\alpha,j}(t)$.}

From Lemma  \ref{approxsequence} there is a constant $C$ and constant $K(\widetilde{\beta})$ such that

\begin{equation}\label{C0esitmatee1}
\frac{\widetilde{\beta}}{2} \log \|S\|^2 - K_{\widetilde{\beta}}\leq \varphi_{\alpha,j}<C\end{equation}
on $M$ for all $\alpha\leq \widetilde{\beta}/2$ and all $j$.  The upper bound follows simply from Lemma \ref{approxsequence} (2) and the fact the $\psi_{\alpha,j}\in C^{\infty}(M)$ is decreasing.  On the other hand, the fact that $\varphi_{\alpha,j}\downarrow\alpha\log\|S\|_{\widehat{h}}^{2}+\varphi_0$, and $\varphi_0$ has zero Lelong number (see definition \ref{psh}) and that $\|S\|_{\widehat{h}}(x) \to 0$ as $x\to D$ in $\overline{M}$, while $\log \|S\|^2=\log \widetilde{c}+\log \|S\|_{\widetilde{h}}^2$ together imply the lower bound in \eqref{C0esitmatee1} some constant $K_{\widetilde{\beta}}>0$ and any $\alpha\leq \widetilde{\beta}/2$ and $j$

\begin{thm}\label{mainC0estimates}
There is a bounded continuous function $U(t)$ on $[0,\widetilde{T}]$ such that
$\varphi_{\alpha,j}(t)\leq U(t)$ on $M\times [0,\widetilde{T}]$ for all $\alpha$ and $j$.   There is a continuous function
$L_{\widetilde{\beta}}(t)$ on $[0,(1-c_{\widetilde{\beta}})\widetilde{T}]$ such that $\frac{3}{2}\widetilde{\beta}\log\|S\|^{2}+L_{\widetilde{\beta}}(t)\leq\varphi_{\alpha,j}(t)$ on $M\times [0,(1-c_{\widetilde{\beta}})\widetilde{T}]$
for all $\alpha\leq\widetilde{\beta}/2$ and all $j$. \end{thm}

 We first prove

\begin{lem}\label{C0estimateprelemma} We have
\begin{enumerate}
\item $\frac{\|S\|^{2}\log^{2}\|S\|^{2}\theta_{t}^{n}}{\Omega}\leq C_{1}(1+t)$
for all $t\in[0,\widetilde{T}]$;
\item $\frac{\|S\|^{2}\log^{2}\|S\|^{2}(\theta_{t}-\widetilde{\beta} \Theta)^{n}}{\Omega}\geq C_{2}t$
for all $t\in[0,(1-c_{\widetilde{\beta}})\widetilde{T}]$,
\end{enumerate}

where the constants $C_{i}>0$ depend on $\Omega$, $\widehat{h}$ and $\widetilde{T}$.

\end{lem}
\begin{proof}

 Now it suffices to show that $C_{2}t\leq\frac{\|S\|^{2}\log^{2}\|S\|^{2}\theta_{t}^{n}}{\Omega}\leq C_{1}+C_{1}t$
for all $t\in[0,\widetilde{T}]$, where $C$ is a constant depending only on $h$
and $\widetilde{T}$, since by (2) Assumption \ref{ass2} we have for all $t\in [0, (1-c_{\widetilde{\beta}})\widetilde{T}]$
\begin{eqnarray*}
\frac{\|S\|^{2}\log^{2}\|S\|^{2}(\theta_{t}-\widetilde{\beta}\Theta)^{n}}{\Omega} & \geq & \frac{(1-c_{\widetilde{\beta}})^{n}\|S\|^{2}\log^{2}\|S\|^{2}\theta_{\frac{t}{1-c_{\widetilde{\beta}}}}^{n}}{\Omega}\\
 & \geq & \frac{1}{2^{n}}\frac{\|S\|^{2}\log^{2}\|S\|^{2}\theta_{\frac{t}{1-c_{\widetilde{\beta}}}}^{n}}{\Omega}.
\end{eqnarray*}
 Now $\theta_{t}=\tilde{\theta}_{t}+\frac{2ti\partial\|S\|^{2}\wedge\bar{\partial}\|S\|^{2}}{\|S\|^{4}\log^{2}\|S\|^{2}}$
where $\tilde{\theta}_{t}=\eta-t\Ric\Omega+t\Theta+\frac{2t\Theta}{\log\|S\|^{2}}.$
 Thus $\theta_{t}^{n}=\tilde{\theta}_{t}^{n}+n\tilde{\theta}_{t}^{n-1}\wedge(\frac{2ti\partial\|S\|^{2}\wedge\bar{\partial}\|S\|^{2}}{\|S\|^{4}\log^{2}\|S\|^{2}})$ and
 $$\|S\|^{2}\log\|S\|^{2}\theta_{t}^{n}=\|S\|^{2}\log^{2}\|S\|^{2}\tilde{\theta}_{t}^{n}+n\tilde{\theta}_{t}^{n-1}\wedge\frac{2ti\partial\|S\|^{2}\wedge\bar{\partial}\|S\|^{2}}{\|S\|^{2}}$$
From this, the fact that $\frac{2i\partial\|S\|^{2}\wedge\bar{\partial}\|S\|^{2}}{\|S\|^{2}}$
is a continuous positive (1,1) form on $\bar{M}$, and the positivity of $\tilde{\theta}_{t}$ on $t\in [0, \widetilde{T}]$ we conclude
$C_{2}t\leq\frac{\|S\|^{2}\log\|S\|^{2}\theta_{t}^{n}}{\Omega}\leq C_{1}+C_{1}t$ on $M\times[0, \widetilde{T}]$
as claimed.
\end{proof}

\begin{proof} [Proof of Theorem \ref{mainC0estimates}]
For all $\alpha \leq \widetilde{\beta}/2$ and $j$, consider $$H_{\epsilon}=\varphi_{\alpha,j}(t)-\int_{0}^{t}\log[C_{1}(1+t)]dt-\epsilon t$$ 
on $M\times[0,\widetilde{T}]$  for any $\e>0$ and $C_{1}$ from Lemma \ref{C0estimateprelemma}. Since
$H_{\epsilon}(x,0)=\varphi_{\alpha,j}$ is bounded above by \eqref{C0esitmatee1} and $|\partial_{t}\varphi_{\alpha,j}(t)|$
and hence $|\partial_{t}H_{\e}|$ is bounded on $M\times[0,\widetilde{T}]$ (by \eqref{pma} and that $\omega_{\alpha,j}(t)$ is a complete bounded curvature solution to \eqref{ckrf}), it follows $H_{\epsilon}$ is bounded above on $M\times[0,\widetilde{T}]$.  Now suppose $H_{\epsilon}$ attains a maximum value on $M\times[0,\widetilde{T}]$ at $(\bar{x}, \bar{t})$.  Then if $\bar{t}>0$, using \eqref{pma} and Lemma  \ref{C0estimateprelemma} we have at $(\bar{x}, \bar{t})$:
\begin{eqnarray*}\label{e1}
\partial_{t}H_{\epsilon} & \leq & \log\frac{\|S\|^{2}\log\|S\|^{2}\theta_{\bar{t}}^{n}}{\Omega}-\log[C_{1}(1+\bar{t})]-\epsilon\leq -\e\\
\end{eqnarray*}
which contradics the maximality assumption.  Thus $\bar{t}=0$ in which case we may simply take $U(t)=C+\int_{0}^{t}\log[C_{1}(1+t)]dt$ for some $C$ by \eqref{C0esitmatee1}.   In general, if $H_{\epsilon}$ does not attain a maximum value on $M\times[0,\widetilde{T}]$ we may argue as in the proof of Lemma \ref{C0-smooth-appr}
and apply the above estimates along an appropriate sequence in space-time (using the Omori-Yau maximum principle) and likewise take $U(t)=C+\int_{0}^{t}\log[C_{1}(1+t)]dt$ for some $C$ in this case as well.

For the lower bound we take $Q_{\epsilon}(x,t)=\varphi_{\alpha,j}(x,t)-\widetilde{\beta}\log|S(x)|^{2}-\int_{0}^{t}\log(C_{2}t)dt+\epsilon t$
on $M\times[0,(1-c_{\widetilde{\beta}})\widetilde{T}]$ for any $\e>0$ and $C_2$ from Lemma
\ref{C0estimateprelemma}.  It follows from \eqref{C0esitmatee1}, \eqref{pma} and that $\omega_{\alpha,j}(t)$ is a bounded curvature solution that $Q_{\epsilon}(x,t)\to\infty$ uniformly as $x$ approaches $D$ on $M\times[0,\widetilde{T}]$ and hence $Q_{\epsilon}(x,t)$ attains an interior minimum on $M\times[0,(1-c_{\widetilde{\beta}})\widetilde{T}]$. Using again Lemma
\ref{C0estimateprelemma}  we may argue as above for the upper bound and conclude that  $L_{\widetilde{\beta}}(t)$ can be taken as $-K_{\widetilde{\beta}}+\int_{0}^{t}\log(C_{2}t)dt$.
\end{proof}


\subsubsection{Local $C^{0}$ estimates of $\dot{\varphi}_{\alpha,j}(t)$.}

\begin{thm} We have
 $\dot{\varphi}_{\alpha,j}(t)\leq\frac{U(t)-\widetilde{\beta}\log\|S\|^{2}+K_{\widetilde{\beta}}}{t}+n$
on $M\times[0,(1-c_{\widetilde{\beta}})\widetilde{T}]$ for all $\alpha\leq\widetilde{\beta}/2$ and all $j$.  \end{thm}
\begin{proof}
The proof is the same as Proposition 3.1 in \cite{GZ}.  Let $H=t\dot{\varphi}_{\alpha,j}(t)-(\varphi_{\alpha,j}(t)-\varphi_{\alpha,j})-nt$.  Then using \eqref{pma}  we have $(\partial_{t}-\Delta)H<0$, where $\Delta$ is the Laplacian with respect to $\omega_{\alpha,j}(t)$. Also, since that $\omega_{\alpha,j}(t)$ is a bounded curvature solution it follows $H$ is a bounded function
on $M\times[0,\widetilde{T}]$, and thus by the maximum principle in \cite{Shi2} we $H \leq \sup_{x\in M}H(x,0)=0$ on $M\times[0,\widetilde{T}]$. Then combining
with Lemma \ref{C0estimateprelemma} and \eqref{C0esitmatee1} we obtain the theorem.
\end{proof}

\begin{thm}
\label{thm:lowerbddphidot} For all $A>0$ with $0<\widetilde{T}-\frac{1}{A}$
, there is a smooth function $F(\|S\|^{2}(x),t)$ on $M\times(0,(1-c_{\widetilde{\beta}})(\widetilde{T}-\frac{1}{A})]$ such that
$\dot{\varphi}_{\alpha,j}(x,t)\geq F(\|S\|^{2}(x),t)$
on $M\times(0,(1-c_{\widetilde{\beta}})(\widetilde{T}-\frac{1}{A})]$ for all  $\alpha\leq 2\widetilde{\beta}$ and all $j$. \end{thm}
\begin{proof}

 For all sufficiently small $\epsilon$, there exists a constant $C>0$ such that \begin{equation}\label{epsilon}\theta_{t}-s\epsilon i\partial\bar{\partial}\log\log^{2}\|S\|^{2}\geq C\widehat{\omega}\end{equation}
on $M\times[0, \widetilde{T}]$ for all $1\leq s\leq 2$

Fix some $A>0$ with $0<\widetilde{T}-\frac{1}{A}$ and $\e$ as above.  Let $Q=\dot{\varphi}_{\alpha,j}(t)+A(\varphi_{\alpha,j}(t)-\widetilde{\beta}\log\|S\|^{2}+\epsilon\log\log^{2}\|S\|^{2})-n\log t$.
By our previous bounds,  $Q\to\infty$ on $M$ as $t\to0$ or $\|S\|\to0$.   In fact, from \eqref{pma} and that $\omega_{\alpha,j}(t)$ is a bounded
curvature solution, $Q(x,t)\to\infty$ uniformly on $M$ as $x$ approaches $D$ for all $t\in[0,(1-c_{\widetilde{\beta}})(\widetilde{T}-\frac{1}{A})]$.  So $Q$ has a minimum on $M\times(0,(1-c_{\widetilde{\beta}})(\widetilde{T}-\frac{1}{A})]$ at some
point $(\bar{x},\bar{t})$ with $\bar{t}>0$. Let $\Delta$ be the Laplacian with respect to $\omega_{\alpha,j}(t)$, using \eqref{pma},  we have

\begin{eqnarray*}(\partial_{t}-\Delta)(\varphi_{\alpha,j}(t)-\widetilde{\beta}\log\|S\|^{2}+\epsilon\log\log^{2}\|S\|^{2}\\
= \dot{\varphi}_{\alpha,j}-n +Tr_{\omega_{\alpha,j}}\left(\theta_{t}-\widetilde{\beta}\Theta-\epsilon i\partial\bar{\partial}\log\log^{2}\|S\|^{2}\right)
\end{eqnarray*}

$$(\partial_{t}-\Delta)\dot{\varphi}_{\alpha,j}(t)=Tr_{\omega_{\alpha,j}}\chi.$$
 Then at $(\bar{x},\bar{t})$ we have the following, where we will use $C$ to denote a constant which is independent of $\alpha$, $j$ and which may differ from line to line.
\begin{eqnarray*}
0 & \geq & (\partial_{t}-\Delta)Q(\bar{x},\bar{t})\\
 & = & ATr_{\omega_{\alpha,j}}(\theta_{\bar{t}+\frac{1}{A}}-\widetilde{\beta}\Theta-\epsilon i\partial\bar{\partial}\log\log^{2}\|S\|^{2})+A\dot{\varphi}_{\alpha,j}-nA-\frac{n}{\bar{t}}\\
 & \geq & A(1-c_{\widetilde{\beta}})Tr_{\omega_{\alpha,j}}[\theta_{\frac{1}{1-c_{\widetilde{\beta}}}(\bar{t}+\frac{1}{A})}-\frac{1}{1-c_{\widetilde{\beta}}}\epsilon i\partial\bar{\partial}\log\log^{2}\|S\|^{2}]+A\dot{\varphi}_{\alpha,j}\\
 & &-nA-\frac{n}{\bar{t}}.\\
 & \geq & AC(1-c_{\widetilde{\beta}})Tr_{\omega_{\alpha,j}}\widehat{\omega}+A\log\frac{\|S\|^{2}\log\|S\|^{2}\omega_{\alpha,j}^{n}}{\Omega}-nA-\frac{n}{\bar{t}}\\
 & \geq & CTr_{\omega_{\alpha,j}\widehat{\omega}}+C\log\frac{\omega_{\alpha,j}^{n}}{\widehat{\omega}^{n}}-\frac{C}{\bar{t}}.\\
 & \geq & CTr_{\omega_{\alpha,j}}\widehat{\omega}-\frac{C}{\bar{t}}\\
 & \geq & C\left(\frac{\widehat{\omega}^{n}}{\omega_{\alpha,j}^{n}}\right)^{\frac{1}{n}}-\frac{C}{\bar{t}}.
\end{eqnarray*}
where we have used Assumption 2 in the third line, $c_{\widetilde{\beta}}\leq\frac{1}{2}$ and \eqref{epsilon} in the fourth line, (3) in \S 2.1 in the fifth line, and the fact $\frac{1}{\lambda}+C\log\lambda$ is bounded below by
some constant depending on $C$ in the sixth line. Therefore, at $(\bar{x},\bar{t})$,
$\omega_{\alpha,j}^{n}\geq C\bar{t}^{n}\widehat{\omega}^{n}$ and so $\dot{\varphi}_{\alpha,j}(\bar{x},\bar{t})\geq C+n\log\bar{t}$ by \eqref{pma} and (3) in \S 2.1.
Since $\log\log^{2}\|S\|^{2}>1$ by Assumption 2,
we have $Q(\bar{x},\bar{t})\geq C+A(\varphi_{\alpha,j}(\bar{x},\bar{t})-\widetilde{\beta}\log|S(\bar{x})|^{2})$.
By Theorem \ref{mainC0estimates}, $\varphi_{\alpha,j}(t)-\widetilde{\beta}\log\|S\|^{2}\geq C$ and
so $Q(\bar{x},\bar{t})\geq C$.  From this, and the upper bound of $\varphi_{\alpha,j}(t)$ from  Theorem \ref{mainC0estimates}, we conclude the lower bound for $\dot{\varphi}_{\alpha,j}(t)$ in the Theorem.
\end{proof}

\subsubsection{Local trace estimates for $\omega_{\alpha, j}(t)$}

Note for all $\alpha$ and $j$, since $\omega_{\alpha,j}(t)$ is
a bounded curvature solution on $M\times[0, T_{[\omega_0 ]})$, so $\omega_{\alpha,j}(t)$ will be uniformly
equivalent to $\widehat{\omega}$ on any closed subinterval of $[0,T_{[\omega_0]})$.
In particular, $Tr_{\widehat{\omega}}\omega_{\alpha,j}(t)$ will be a
bounded function on $M\times[0, T]$
\begin{thm}
There is a smooth function $G(\|S\|^{2}(x),t)$ on $M\times(0,(1-c_{\widetilde{\beta}})\widetilde{T}]$
such that $Tr_{\widehat{\omega}}\omega_{\alpha,j}(x,t)\leq G(\|S\|^{2}(x),t)$
for all $2\alpha\leq\widetilde{\beta}$ and all $j$.\end{thm}
\begin{proof}
 Consider $Q(\cdot,t)=t\log Tr_{\widehat{\omega}}\omega_{\alpha,j}(t)-B(\varphi_{\alpha,j}(t)-\widetilde{\beta}\log\|S\|^{2}+\epsilon\log\log^{2}\|S\|^{2})$,
where $\epsilon$ is chosen as in \eqref{epsilon}  and $B>0$ is a large 
constant which will be determined later, independently
of $\alpha,j$. Now $Q(x,0)\to-\infty$
as $x$ approaches $D$ from (4.9), and from \eqref{pma} and that $\omega_{\alpha,j}(t)$ is a bounded
curvature solution, $Q(x,t)\to-\infty$ uniformly as $x$ approaches
$D$ for all $t\in[0,(1-c_{\widetilde{\beta}})\widetilde{T}]$.  Hence $Q(x,t)$ attains a maximum on  $M\times[0,(1-c_{\widetilde{\beta}})\widetilde{T}]$ at some
point $(\bar{x},\bar{t})$.   In the following, $C_i$'s will denote  positive constants independent
of $\alpha,j$. 

If $\bar{t}>0$, then $0\leq(\partial_{t}-\Delta)Q(\bar{x},\bar{t})$, where $\Delta$ is the Laplacian with respect to $\omega_{\alpha,j}(t)$.
Also, we have \begin{equation}(\partial_{t}-\Delta)\log Tr_{\widehat{\omega}}\omega_{\alpha,j}(t)\leq C_1Tr_{\omega_{\alpha,j}}\widehat{\omega}\end{equation}
for some constant $C_1$ depending only on $\widehat{\omega}$ (see \cite{ST0}), so
\begin{eqnarray*}
(\partial_{t}-\Delta)t\log Tr_{\widehat{\omega}}\omega_{\alpha,j}(t) & = & \log Tr_{\widehat{\omega}}\omega_{\alpha,j}(t)+t(\partial_{t}-\Delta)\log Tr_{\widehat{\omega}}\omega_{\alpha,j}(t)\\
 & \leq & \log Tr_{\widehat{\omega}}\omega_{\alpha,j}(t)+C_1 tTr_{\omega_{\alpha,j}(t)}\widehat{\omega}.
\end{eqnarray*}
Using the computations in the proof of Theorem \ref{thm:lowerbddphidot}, we have
\[
 (\partial_{t}-\Delta)(\varphi_{\alpha,j}(t)-\widetilde{\beta}\log\|S\|^{2}+\epsilon\log\log^{2}\|S\|^{2})
\geq\dot{\varphi}_{\alpha,j}-n+(1-c_{\widetilde{\beta}})C_2Tr_{\omega_{\alpha,j}}\widehat{\omega}.
\]


Therefore at $(\bar{x},\bar{t})$, using that $c_{\beta}<1/2$ from Assumption 2, we have
\[
\begin{split}
0\leq (\partial_{t}-\Delta)Q\leq&\log Tr_{\widehat{\omega}}\omega_{\alpha,j}-B\dot{\varphi}_{\alpha,j}+nB+(C_1t-\frac{1}{2}BC_2)Tr_{\omega_{\alpha,j}}\widehat{\omega}\\
\leq&\log Tr_{\widehat{\omega}}\omega_{\alpha,j}-Tr_{\omega_{\alpha,j}}\widehat{\omega}-B\dot{\varphi}_{\alpha,j}+nB
\end{split} \]
 where in the second line we have assumed a choice $B$, independent
of $\alpha, j$ and $\bar{t}$, such that $C_1\widetilde{T}-\frac{1}{2}BC_2<-1$.

Since $Tr_{\widehat{\omega}}\omega_{\alpha,j}\leq(Tr_{\omega_{\alpha,j}}\widehat{\omega})^{n-1}\frac{\omega_{\alpha,j}^{n}}{\widehat{\omega}^{n}}$
and $\dot{\varphi}_{a,j}\geq C_{3}\log\frac{\omega_{\alpha,j}^{n}}{\widehat{\omega}^{n}}$
for some $C_{3}$ depending only on $h$ and $\widehat{\omega}$, putting
them into the above expression, we get
\[
0\leq(n-1)\log Tr_{\omega_{\alpha,j}}\widehat{\omega}+(1-BC_{3})\log\frac{\omega_{\alpha,j}^{n}}{\widehat{\omega}^{n}}-Tr_{\omega_{\alpha,j}}\widehat{\omega}+C_4.
\]
Assume further that $BC_{3}>2$, we have
\begin{eqnarray*}
0 & \leq & -Tr_{\omega_{\alpha,j}}\widehat{\omega}+(n-1)\log Tr_{\omega_{\alpha,j}}\widehat{\omega}-\log\frac{\omega_{\alpha,j}^{n}}{\widehat{\omega}^{n}}+C_4\\
 & \leq & -\frac{1}{2}Tr_{\omega_{\alpha,j}}\widehat{\omega}-\log\frac{\omega_{\alpha,j}^{n}}{\widehat{\omega}^{n}}+C_4,
\end{eqnarray*}
we used $-x+C\log x$ is bounded above for $x>0$ by some constant depending
on $C$.  Now let $\lambda_i$ be the eigenvalue of $\omega_{\alpha,j}(\bar{x}, \bar{t})$ relative to $\widehat{\omega}(\bar{x}, \bar{t})$ and let $C$ denote a postive constant independent of $\alpha, j$ which may differ from line to line.  Then the previous equation says
$$\sum_i (\frac{1}{2\lambda_i}+\log\lambda_i)\leq C$$
and from the fact that the function $1/2x+\log x$ is bounded below for all $x>0$, we get that $ (\frac{1}{2\lambda_i}+\log\lambda_i)\leq C$ for each $i$, thus $Tr_{\widehat{\omega}}\omega_{\alpha,j}(\bar{x}, \bar{t})\leq C$. Since $\varphi_{\alpha,j}(t)-\widetilde{\beta}\log\|S\|^{2}\geq C$
and $\epsilon\log\log^{2}\|S\|^{2}\geq0$ we conclude $Q(\bar{x},\bar{t})\leq C$.  Thus by our earlier observed upper bound for $Q(x, 0)$ we get $Q(x, t)\leq C$ on  $M\times[0,(1-c_{\widetilde{\beta}})\widetilde{T}]$ and the Theorem follows from this and the upper bound in Theorem \ref{mainC0estimates}.
\end{proof}

\subsection{Completion of Proof of Theorem \ref{T1}}

 Now recall our family $\omega_{\alpha,j}(t)$ of solutions to \eqref{ckrf} on $M\times[0, T_{[\omega_0 ]})$ from Lemma \ref{LZ}.
Recall that we wrote $\omega_{\alpha,j}(t)=\theta_t+ i\ddbar\varphi_{\alpha,j}(t)$ where $\varphi_{\alpha,j}(t)$ solves \eqref{pma} on $M\times[0, T_{[\omega_0 ]})$.  Also recall the choices made in Assumption \ref{ass2}, and in particular that $0<T<\widetilde{T}<T_{[\omega_0 ]}$ was arbitrary.

 From the Theorem 4.3, 4.4 and \eqref{pma}, for any $s>0$ and compact subsets  $K_1\subset \subset K_2 \subset \subset M$ we may have
\begin{equation}
C_1\eta \leq \omega_{\alpha,j}(t) \leq C_2 \eta
\end{equation}
on $K_2\times[s, T]$ for some constants $C_i$ independent over all $\alpha$ sufficiently small and all $j$ sufficeintly large depending on $\tilde{T}$.  It follows from this and the estimates from the Evans-Krylov theory (see also \cite{WS} for a maximum principle proof of these for \eqref{ckrf}), that for some $\alpha_k \to 0$, $j_k\to \infty$, we have $\omega_{\alpha_k, j_k}(t)$ converges on $K_1 \times[s, T]$ smoothly to a limit solution $\omega(t)$ to the flow in equation \eqref{ckrf}.   As $T< T_{[\omega_0 ]}$ was chosen arbitrarily, by a diagonal argument we may in fact assume $\omega_{\alpha_k, j_k}(t)$ converges on $M\times(0,  T_{[\omega_0 ]})$, smoothly on compact subsets, to a limit solution $\omega(t)$ to the flow in \eqref{ckrf}, while also $\omega_{\alpha_k, j_k}(0) \to \omega_0$ smoothly on compact subsets of $M$.  By applying Theorem 4.1 in \cite{Chau-Li-Tam} to the sequence $\omega_{\alpha_k, j_k}(t)$ on $M\times[0, T_{[\omega_0 ]})$ and observing the uniform lower bound on $\omega_{\alpha_k, j_k}(0) \geq \delta \widehat{\omega}$ in \eqref{lowerbound}, we see that $\omega_{\alpha_k, j_k}(t)$ actually converges smoothly on $M\times[0,  T_{[\omega_0 ]})$ to a limit solution satisfying \eqref{completeness0}.  In other words the solution $\omega(t)$ extends smoothly on $M\times[0,  T_{[\omega_0 ]})$ and satisfies \eqref{completeness0}.

 We now show (1) in the Theorem \ref{T1}  is satisfied.  Fix any Hermitian metric $h$ on $\mathcal{O}_D$ and smooth volume form $\Omega$ on $\overline{M}$.  Then as in our derivation of  \eqref{pma} we see that  $\varphi(t):=\varphi_0+ \int_0^t \log  \displaystyle \log\frac{\|S\|_{h}^{2}\log\|S\|_{h}^{2}(\omega(t))^{n}}{\Omega}$ solves
\eqref{pma0} on $M\times[0, T_{[\omega_0 ]})$ and \eqref{krfansatz}.   In particular,  $\varphi(t)=\displaystyle \lim_{k\to \infty} u_{\alpha_k, j_k}(t)$ where $u_{\alpha_k, j_k}(t) =\varphi_{\alpha_k, j_k}+ \int_0^t \log  \displaystyle \log\frac{\|S\|_{h}^{2}\log\|S\|_{h}^{2}(\omega_{\alpha_k, j_k}(t))^{n}}{\Omega}$ and $u_{\alpha_k, j_k}(t)$ solves \eqref{pma0} on $M\times[0, T_{[\omega_0 ]})$ with initial data $\varphi_{\alpha_k, j_k}$.  To see that the upper bound in (1) holds, note that the estimate in Lemma \ref{C0estimateprelemma} (1) in fact holds for any, and hence our, choice of $h$ for some constant $C_1$.  Then from the proof of the upper bound in Theorem \ref{mainC0estimates}, there exists a continuous function $U(t)$ such that $u_{\alpha_k, j_k}(t)\leq U(t)$ and hence $\varphi(t)\leq U(t)$ on $M\times[0, T_{[\omega_0 ]})$.  This completes the proof of (1) in the Theorem.

Finally, we show that (2) holds.  Let $\varphi_0$  be as in (2). For any choice of $0<T<\widetilde{T}<T_{[\omega_0 ]}$ and corresponding subsequent choices in Assumption 2,  consider solutions $\varphi_{\alpha, j}(t)$ to \eqref{pma} on $M\times[0, T)$ constructed in the proof of Theorem \ref{T1} so far.  Now if  $\varphi_0$ also satisfies the lower bound in (2) then we may replace the estimates in  \eqref{C0esitmatee1} with
\begin{equation}\label{C0esitmatee22}
\alpha \log \|S\|^2 -C\log\log\|S\|^2 \leq \varphi_{\alpha,j}<C\end{equation}
for some $C$ and all $\alpha\leq \hat{\alpha}$ and all $j$ where $\hat{\alpha}$ is from Lemma 4.1.

Now for $\hat{\alpha}$ sufficiently small, observe that the estimate in Lemma \ref{C0estimateprelemma} (2) still holds after replacing $\widetilde{\beta}$ with any $\alpha\leq \hat{\alpha}$.   Now repeating the proof of the lower bound in Theorem \ref{mainC0estimates}, but using instead the function $Q_{\epsilon}(x,t)=\varphi_{\alpha,j}(x,t)-2\alpha\log\|S(x)\|^{2}-\int_{0}^{t}\log(C_{2}s)ds+\epsilon t$, we may have $\varphi_{\alpha, j}(t)\geq - \alpha\log\|S\|^2 - C\log\log\|S\|^2+ \int_0^t \log (C_2 s) ds$ on $M\times[0, T]$ for all $\alpha\leq \hat{\alpha}$ and all $j$.   The a priori estimates in \S 4 imply that  $\varphi_{\alpha_k, j_k}(t)$ converges smoothly on compact subsets of  $M\times[0, T]$ to some $\varphi(t)$ satisfying the bounds in (2).   This completes the proof of Theorem \ref{T1} (2).

\section{Proof of Theorem \ref{T3}}
We begin with the following Theorem from which Theorem \ref{T3} will follow.  In the following, for any complete \K manifold $(M, \omega)$ with bounded curvature, we use $T(\omega)$ to denote the maximal existence time of a complete bounded curvature solution to the \KRF \eqref{ckrf} starting from $\omega$.  Also, we say $\gamma(x)$ is a distance like function on $(M, \omega)$ if for some $p\in M$ and $C_1, C_2 >0$ we have $C_1^{-1} d(p,x) \leq \gamma(x)\leq C_1 d(p, \cdot)$ whenever $d(p, x)>C_2$, where $d(p, \cdot)$ is the distance function from $p$ on $(M, \omega)$.  We begin by proving the following

 \begin{thm} \label{extime}
Let $(M,\widehat{\omega})$ be a complete \K manifold with bounded curvature. Let $\gamma:M\to \R$ be a smooth distance-like function with $|\hat{\n} \gamma |_{\widehat{\omega}}<C$ and $|i\ddbar \gamma|_{\widehat{\omega}}<C$ for some constant $C$. Let $\varphi\in C^{\infty}(M)$ such that $|\varphi|/\gamma \to 0$ and $|\hat{\n} \varphi|_{\widehat{\omega}}/\gamma \to 0$ as $\gamma\to 0$. If $\omega=\widehat{\omega}+i\ddbar \varphi$ is a complete metric with bounded curvature and satisfies $|\omega-\widehat{\omega}|_{\widehat{\omega}}\to 0$ as $\gamma \to 0$, then $T(\omega)=T(\widehat{\omega})$.
 \end{thm}

 \begin{proof}
Let $\rho:\R \to \R$ be a smooth function such that $\rho=1$ on $[0,1]$ and $\rho=0$ on $[2,\infty)$. Define $\rho_R:M\to \R$ by $\rho_R=\rho(\gamma/R)$ and let $\omega_R=\widehat{\omega}+i\ddbar(\rho_R \varphi)$. We claim that if $R$ is sufficiently large then $\omega_R$ is a complete \K metric and there exists $C_R\to 1$ as $R\to \infty$ such that $\frac{1}{C_R} \omega \leq \omega_R\leq C_R \omega$.

We have $\omega_R=\rho_R \omega+(1-\rho_R)\widehat{\omega}+2\re(i\p \rho_R \wedge \bar{\p}\varphi)+i\varphi \ddbar\rho_R$. Since $|\omega-\widehat{\omega}|_{\widehat{\omega}}\to 0$ as $\gamma \to \infty$, we have $\frac{1}{C_R} \omega \leq \rho_R \omega+(1-\rho_R)\widehat{\omega}\leq C_R \omega$ for some $C_R \to 1$ as $R\to \infty$. Now it suffices to show that $|2\re(i\p \rho_R \wedge \bar{\p}\varphi)|_{\widehat{\omega}}\to 0$  and $|i\varphi \ddbar\rho_R|_{\widehat{\omega}}\to 0$ uniformly on $M$ as $R\to \infty$.

For any point in $M$, we have
\begin{align*}
	|2\re(i\p \rho_R \wedge \bar{\p}\varphi)|_{\widehat{\omega}}& \leq|\frac{\rho'(\frac{\gamma}{R})}{R}\p \gamma \wedge \bar{\p} \varphi|_{\widehat{\omega} }\\
	& \leq \frac{|\rho'(\frac{\gamma}{R})|}{R}|\hat{\nabla} \gamma|_{\widehat{\omega}} |\hat{\nabla} \varphi|_{\widehat{\omega}}\\
	& \leq \frac{ \displaystyle C(\max_{\R} |\rho'|)\chi_{\gamma^{-1}{[R,2R]}}}{R}|\n \varphi|_{\widehat{\omega}}\\
	& \leq 2C(\max_{\R}  |\rho'|)\chi_{\gamma^{-1}{[R,2R]}}\frac{|\n \varphi|_{\widehat{\omega}}}{\gamma}.\\
\end{align*}

Because $|\hat{\nabla} \varphi|_{\widehat{\omega}}/\gamma \to 0$ as $\gamma \to \infty$, the function on the right hand side converges uniformly to 0 as $R\to \infty $. Similar argument works for
$$|i\varphi \ddbar\rho_R|_{\widehat{\omega}}=|\varphi\rho'(\frac{\gamma}{R})\frac{i\p \bar{\p}\gamma}{R}+\varphi\rho''(\frac{\gamma}{R})\frac{i\p\gamma\wedge \bar{\p}\gamma}{R^2}|_{\widehat{\omega}}.$$

Therefore, we have a family of complete \K metrics $\omega_R$ such that  $\frac{1}{C_R} \omega \leq \omega_R\leq C_R \omega$ with $C_R\to 1$ as $R\to \infty$ and it is clear that $\omega_R$ has bounded curvature. Therefore, by Theorem 2.2 in \cite{Chau-Li-Tam2}, we have  $\frac{1}{C_R} T(\omega_R)\leq T(\omega) \leq C_R T(\omega_R)$.  On the other hand, since $\rho_R\varphi$ has compact support, by Theorem 4.1 in \cite{LZ}, we have $T(\omega_R)=T(\widehat{\omega})$ for all $R$. Therefore, passing the limit $R\to \infty$ we obtain $T(\omega)=T(\widehat{\omega})$.
 \end{proof}

\begin{proof}[Proof of Theorem \ref{T3}] The uniqueness of bounded curvature solutions follows from \cite{Chen-Zhu}.  Let $p\in M$ and let $d_{\widehat{\omega}}(p,\cdot)$ be the distance function to $p$ relative to $\widehat{\omega}$.  Let $\gamma(x):= \log\log^2|S(x)|^2$ on $M$.  Then from  \eqref{CGmetric}  we may write
$$\widehat{\omega}=\bar{\eta} - i\ddbar \gamma(x)=\overline{\eta} -2\frac{dd^c \log \|S\|^{2}_{h}}{\log \|S\|^{2}_{h}}+2 i\partial \gamma \wedge \bar{\partial}\gamma.$$  Noting that $\eta$ as well as the numerator of the second term above are smooth forms on $\overline{M}$,  we see that for all $x\in M$ sufficiently close to $D$, or equivalently when $\gamma(x)$ is sufficiently large, we have $ C^{-1} \gamma(x)\leq d_{\widehat{\omega}}(p,x)\leq C \gamma(x)$ and $\| d \gamma(x) \|_{\widehat{\omega}}<C$ for some constant $C$.   Moreover, for all $x\in M$ sufficiently close to $D$ we also see from above that $- i\ddbar \gamma(x) >0$, and from this and the first equality above we may conclude that  $\|i\ddbar \gamma(x)\|_{\widehat{\omega}}\leq C$ for some $C$ independent of $x$.  In other words, $\gamma$ satisfies the assumption in Theorem \ref{extime} relative to $\widehat{\omega}$, and Theorem follows immediately.
\end{proof}

\begin{rem}
In Theorem \ref{T3} we can remove the condition on $d\varphi$ if we assume $\omega_0$ has the same \begin{it} standard spatial asymptotics\end{it} as that of $\widehat{\omega}$ as defined in \cite{LZ}.  As an example, if $\omega=\widehat{\omega}+i\ddbar \log\log\log^2\|S\|^2$ defines a metric, then it has standard spatial asymptotics at $D$ but not superstandard spatial asymptotics (see example 8.12 in \cite{LZ}) while Theorem \ref{T3} still provides a bounded curvature solution on $M\times[0, T_{[\omega_0]})$.
\end{rem}

\end{document}